\documentclass[11pt]{article}
\usepackage{bbm}
\usepackage{geometry}
\usepackage{color}
\usepackage{amsfonts}
\usepackage{algorithm}
\usepackage{algorithmic}
\usepackage{latexsym, amssymb, amsmath, amscd, amsthm, amsxtra}
\usepackage{mathtools}
\usepackage{enumerate}
\usepackage[all]{xy}
\usepackage{mathrsfs}
\usepackage{fancyhdr}
\usepackage{listings}
\usepackage{hyperref}
\usepackage{enumitem}

\def\P{\mathbb{P}}

\def\E{\mathbb{E}}

\def\11{\mathbbm{1}}
\def\Gc{\mathcal{G}}

\def\Es{\mathsf E}
\def\ER{Erd\H{o}s-R\'enyi\ }

\def\Vs{\mathsf V}
\def\Gs{\mathsf G}

\newcommand{\Pb}{\mathbb P}

\newtheorem{thm}{Theorem}[section]

\newtheorem{proposition}[thm]{Proposition}

\newtheorem{lemma}[thm]{Lemma}

\newtheorem{defn}[thm]{Definition}

\theoremstyle{definition}
\newtheorem{remark}[thm]{Remark}

\numberwithin{equation}{section}

\makeatletter
\newenvironment{breakablealgorithm}
{
		\begin{center}
			\refstepcounter{algorithm}
			\hrule height.8pt depth0pt \kern2pt
			\renewcommand{\caption}[2][\relax]{
				{\raggedright\textbf{\ALG@name~\thealgorithm} ##2\par}%
				\ifx\relax##1\relax 
				\addcontentsline{loa}{algorithm}{\protect\numberline{\thealgorithm}##2}%
				\else 
				\addcontentsline{loa}{algorithm}{\protect\numberline{\thealgorithm}##1}%
				\fi
				\kern2pt\hrule\kern2pt
			}
		}{
		\kern2pt\hrule\relax
	\end{center}
}
\makeatother

\hypersetup{colorlinks=true,linkcolor=blue}

\begin{document}

	\title{A polynomial-time approximation scheme for the maximal overlap of two independent \ER graphs}
	
	\author{Jian Ding \\ Peking University\and Hang Du \\ Peking University\and Shuyang Gong\\ Peking University}
	
	\maketitle

\begin{abstract}
	For two independent \ER graphs $\mathbf G(n,p)$, we study the maximal overlap (i.e., the number of common edges) of these two graphs over all possible vertex correspondence. We present a polynomial-time algorithm which finds a vertex correspondence whose overlap approximates the maximal overlap up to a multiplicative factor that is arbitrarily close to 1. As a by-product, we prove that the maximal overlap is asymptotically $\frac{n}{2\alpha-1}$ for $p=n^{-\alpha}$ with some constant $\alpha\in (1/2,1)$. 
\end{abstract}

\section{Introduction}\label{sec:intro}

In this paper we study the random optimization problem of maximizing the overlap for two independent \ER graphs over all possible vertex correspondence. More precisely, we fix two vertex sets $V=\{v_1,\dots, v_n\},\Vs=\{\mathsf v_1,\dots,\mathsf v_n\}$ with edge sets 
\begin{equation}
\label{def:E0 and Es0}
	E_0=\{(v_i,v_j):1\le i<j\le n\},\ \Es_0=\{(\mathsf v_i,\mathsf v_j):1\le i<j\le n\}\,.
\end{equation}
For a parameter $p\in (0,1)$, we consider two \ER random graphs $ G = (V,E)$ and $\Gs = (\Vs,\Es)$ where $E \subset E_0$ and $\Es\subset \Es_0$ are obtained from keeping each edge in $E_0$ and $\Es_0$ (respectively) with probability $p$ independently. For convenience, we abuse the notation and denote by $G$ and $\Gs$ the adjacency matrices for the two graphs respectively. That is to say, $G_{i, j} = 1$ if and only if $(v_i, v_j)\in E$ and $\Gs_{i, j} = 1$ if and only if $(\mathsf v_i, \mathsf v_j) = 1$ for all $1\leq i<j\leq n$.
 Let $S_n$ be the collection of all permutations on $[n]=\{1, \ldots, n\}$.  For $\pi\in S_n$, define
 \begin{equation}
	\mathrm{Overlap}(\pi)\stackrel{\text{def}}{=}\sum_{1\leq i<j\leq n} G_{i,j}\Gs_{\pi(i),\pi(j)}
\end{equation}
to be the number of edges $(v_i, v_j)\in E$ with $(\mathsf v_{\pi(i)}, \mathsf v_{\pi(j)})\in \Es$.
Our main result is a polynomial-time approximation scheme (PTAS) for this optimization problem.
\begin{thm}
	\label{thm:main}
	Assume $p=n^{-\alpha}$ for some $\alpha\in (1/2,1)$. For any fixed $\varepsilon>0$, there exist a constant $C=C(\varepsilon)$ and an algorithm (see Algorithm~\ref{Alg}) with running time $O(n^C)$ which takes $G,\Gs$ as input and outputs a permutation $\pi^*\in S_n$, such that 
	\begin{equation}
		\label{eq:approximation}
		\P\left[\frac{\mathrm{Overlap}(\pi^*)}{n}\ge \frac{1-\varepsilon}{2\alpha-1}\right]=1-o(1),\quad\text{as }n\to\infty. 
	\end{equation}
	In addition, $\max_{\pi\in S_n} \frac{\mathrm{Overlap}(\pi)}{n}$ converges to $\frac{1}{2\alpha - 1}$ in probability as $n\to \infty$.
\end{thm}
\begin{remark}
For $\alpha \in (0, 1/2)$, a straightforward computation yields that with probability tending to 1 as $n\to \infty$, the overlap is asymptotic to $\frac{n^2 p^2}{2}$ for all $\pi \in S_n$. Thus, this regime is trivial if one's goal is to approximate the maximal overlap in the asymptotic sense. For the ``critical'' case when $\alpha$ is near $1/2$, the problem seems more delicate and our current method falls short in approximating the maximal overlap asymptotically.
\end{remark}
\begin{remark}
One can find Hamiltonian cycles in both graphs with an efficient algorithm (see \cite{NSS21} as well as references therein) and obviously this leads to an overlap with $n$ edges. Also, a simple first moment computation as in the proof of Proposition~\ref{prop-upper-bound} shows that the maximal overlap is of order $n$. Therefore, the whole challenge in this paper is to nail down the asymptotic constant $1/(2\alpha - 1)$.
\end{remark}
\begin{remark}\label{rem-Xu-Wu}
Theorem~\ref{thm:main} provides an \emph{algorithmic} lower bound on the maximal overlap which is sharp asymptotically (i.e., the correct order with the correct leading constant), and we emphasize that prior to our work the asymptotic maximal overlap was not known even from a non-constructive perspective. In fact, the problem of the asymptotic maximal overlap was proposed by  Yihong Wu and Jiaming Xu as an intermediate step in analyzing the hardness of matching correlated random graphs, and Yihong Wu and Jiaming Xu have discussed extensively with us on this problem. It is fair to say that we have come up with a fairly convincing roadmap using the second moment method, but the technical obstacles seem at least somewhat daunting such that we did not manage to complete the proof. The current paper was initiated with the original goal of demonstrating an \emph{information-computation gap} for this random optimization problem (as we have believed back then), but it turned out to evolve in a somewhat unexpected way: not only the maximal overlap can be approximated by a polynomial time algorithm, but also it seems (as of now) this may even be a more tractable approach even if one's goal is just to derive the asymptotic maximal overlap. Finally, the very stimulating discussions we had with Yihong Wu and Jiaming Xu, while not explicitly used in the current paper, have been hugely inspiring to us. We record this short ``story'' here as on the one hand it may be somewhat enlightening to the reader, and on the other hand we would like to take this opportunity to thank Yihong Wu and Jiaming Xu in the warmest terms. 
\end{remark}
\subsection{Background and related results}

Our motivation is twofold: on the one hand, we wish to accumulate insights for understanding the computational phase transition for matching two correlated random graphs; on the other hand, as mentioned in Remark~\ref{rem-Xu-Wu} we wished to study the computational transition for the random optimization problem of the maximal overlap, which is a natural and important combinatorial optimization problem known to be NP-hard to solve in the worst-case scenario. We will further elaborate both of these two points in what follows. 

Recently, there has been extensive study on the problem of matching the vertex correspondence between two correlated graphs and the closely related problem of detecting the correlation between two graphs. From the applied perspective, some versions of graph matching problems have emerged 
from various applied fields such as social network analysis \cite{NS08,NS09}, computer vision \cite{CSS07,BBM05}, computational biology \cite{SXB08,VCP15} and natural language processing \cite{HNM05}. From the theoretical perspective,  graph matching problems seem to provide another important set of examples with the intriguing \emph{information-computation gap}. The study of information-computation gaps for high-dimensional statistical problems as well as random optimization problems is currently an active and challenging topic. For instance, there is a huge literature on the problem of recovering communities in stochastic block models (see the monograph \cite{Abbe17} for an excellent account on this topic) together with some related progress on the optimization problem of extremal cuts for random graphs \cite {DMS17, MS16}; there is also a huge literature on the hidden clique problem and the closely related problem of submatrix detection (see e.g. \cite{wu_xu_2021} for a survey and see e.g. \cite{GZ19+, MRS21} and references therein for a more or less up-to-date review). In addition, it is worth mentioning that in the last few decades various framework has been proposed to provide evidence on computational hardness for such problems with random inputs; see surveys \cite{ZK16, BPW18, RSS19, KWB19, Garmarnik21} and references therein. 

As for random graph matching problems, so far most of the theoretical study was carried out for Erd\H{o}s-R\'enyi graph models since Erd\H{o}s-R\'enyi graph is arguably the most canonical random graph model. Along this line, much progress has been made recently, including information-theoretic analysis \cite{CK16, CK17, HM20, WXY20+,WXY21+,DD22+,DD22+b} and proposals for various efficient algorithms  \cite{PG11, YG13, LFP14, KHG15, FQRM+16, SGE17, BCL19, DMWX21, BSH19, CKMP19,DCKG19, MX20, FMWX20,  GM20, FMWX19+, MRT21+, MWXY21+}.  As of now, it seems fair to say that we are still relatively far away from being able to completely understand the phase transition for computational complexity of graph matching problems---we believe, and we are under the impression that experts on the topic also believe, that graph matching problems do exhibit information-computation gaps. Thus, (as a partial motivation for our study) it is plausible that understanding the computational aspects for maximizing the overlap should shed lights on the computational transition for graph matching problems.

 As hinted earlier, it is somewhat unexpected that a polynomial time approximation scheme for the maximal overlap problem exists while the random graph matching problem seems to exhibit an information-computation gap. As a side remark, before finding this approximation algorithm, we have tried to demonstrate the computational hardness near the information threshold via the overlap gap property but computations suggested that this problem does not exhibit the overlap gap property. In addition, we point out that efficient algorithms have been discovered to approximate ground states for some spin glass problems \cite{Subag21, Montanari19}, which take advantage of the so-called full replica symmetry breaking property. From what we can tell,  \cite{Subag21, Montanari19} seem to be rare natural examples for which approximation algorithms were discovered for random instances whereas the worse-case problems are known to be hard to solve. In a way, our result contributes yet another example of this type and possibly of a different nature since we do not see full replica symmetry breaking in our algorithm.

\subsection{An overview of our method}\label{sec:overview}

The main contribution of our paper is to propose and analyze the algorithm as in Theorem~\ref{thm:main}. The basic idea for our algorithm is very simple: roughly speaking, we wish to match vertices of two graphs sequentially such that the increment on the number of common edges (within matched vertices) per each matched pair is $(2\alpha-1)^{-1}$.  

An obvious issue is that $(2\alpha-1)^{-1}$ may not be an integer. Putting this aside for a moment, we first explain our idea in the simple case when $p=n^{-\frac{3}{4}+\delta}$ with some arbitrarily small $\delta>0$. In this case, our goal is to construct some $\pi\in S_n$ such that $\mathrm{Overlap}(\pi)\ge (2-\varepsilon)n$ for an arbitrarily small $\varepsilon >0$. Assume we have already determined $\pi(1),\ldots,\pi(k)$ for some $k\in \left[\frac\varepsilon 3 n,\left(1-\frac{\varepsilon}{3}\right)n\right]$ (i.e., assume we have already matched $ v_1,\dots,v_{k}\in V$ to $\mathsf v_{\pi(1)},\dots,\mathsf v_{\pi(k)}\in \Vs$). We wish to find some $\ell \in [n]\setminus \{\pi(1),\dots,\pi(k)\}$ such that 
\begin{equation}\label{eq-alg-feasible}
\sum_{1\le j\le k}G_{j,k+1}\Gs_{\pi(j), \ell}\ge 2\,.
\end{equation}
 Provided that this is feasible, we may set $\pi(k+1)= \ell$. Therefore, the crux is to show that \eqref{eq-alg-feasible} is feasible for most of the steps. Ignoring the correlations between different steps for now, we can then regard $\sum_{1\le j\le k} G_{j,k+1}\Gs_{\pi(j), \ell}$ as a Binomial variable $\mathbf B(k, p^2)$ for each  $\ell \in [n]\setminus \{\pi(1),\dots,\pi(k)\}$ and thus \eqref{eq-alg-feasible} holds with probability of order $(k p^2)^2 \gtrsim n^{\delta - 1}$. Since there are at least $\varepsilon n$ potential choices for $\ell$, (ignoring the potential correlations for now again) it should hold with overwhelming probability  that there exists at least one $\ell$ satisfying \eqref{eq-alg-feasible}. This completes the heuristics underlying the success of such a simple iterative algorithm. 
 
The main technical contribution of this work is to address the two issues we ignored so far: the integer issue and the correlations between iterations. In order to address the integer issue, it seems necessary to match $\chi$ vertices every step with an increment of $\zeta$ common edges for some suitably chosen $\chi, \zeta$ with $\zeta/\chi \approx (2\alpha - 1)^{-1}$. Thus, one might think that a natural analogue of \eqref{eq-alg-feasible} shall be the following: there exist $\ell_1, \ldots, \ell_{\chi} \in [n]\setminus \{\pi(1),\dots,\pi(k)\}$ such that
\begin{equation}\label{eq-alg-feasible-general}
\sum_{1\le j\leq k}\sum_{1 \leq i\leq \chi} G_{j,k+i}\Gs_{\pi(j), \ell_i}\geq  \zeta\,.
\end{equation}
However, \eqref{eq-alg-feasible-general} is not really feasible since in order for \eqref{eq-alg-feasible-general} to hold it is necessary that for some $i\in \{1,\dots,\chi\}$, there exists $\ell \in [n]\setminus \{\pi(1),\dots,\pi(k)\}$ such that  $\sum_{1\le j\le k} G_{j,k+i}\Gs_{\pi(j), \ell} \geq \lceil(2\alpha - 1)^{-1} \rceil$ (where $\lceil x\rceil$ denotes the minimal integer that is at least $x$). A simple first moment computation suggests that the preceding requirement cannot be satisfied for most steps. 

In light of the above discussions, we see that in addition to common edges between the matched vertices and the ``new'' vertices, it will be important to also take advantage of common edges that are within ``new'' vertices, which requires to also carefully choose a collection of vertices from $ V$ (otherwise typically there will be no edges within a constant number of fixed vertices). In order to implement this, it turns out that we may carefully construct a rooted tree $\mathbf T$ with $\chi$ non-leaf vertices and $\zeta$ edges, and we wish that the collection of added common edges per step contains an isomorphic copy of $\mathbf T$. To be more precise, let $\operatorname M = \{j\in [n]:   v_j \mbox{ has been matched}\}$ and let $\pi(\operatorname M) = \{\pi(j): j\in \operatorname M\}$, and in addition let $k$ be the minimal integer in $[n]\setminus \operatorname M$. We then wish to strengthen \eqref{eq-alg-feasible-general} to the following: there exists $\xi = \zeta + 1 - \chi$ integers $i_1, \ldots, i_{\xi} \in \operatorname M$, $\chi$ integers $j_1 = k, j_2, \ldots, j_{\chi} \in [n] \setminus \operatorname M $  and  $\chi$ integers $\ell_1, \ldots, \ell_{\chi}\in [n]\setminus \{\pi(1),\dots,\pi(k)\}$ such that the following hold:
\begin{itemize}
\item the subgraph of $G$ induced on $\{v_{i_1}, \ldots, v_{i_{\xi}}\} \cup \{v_{j_1}, \ldots,   v_{j_\chi}\}$ contains $\mathbf T$ as a subgraph with leaf vertices $\{v_{i_1}, \ldots,   v_{i_{\xi}}\}$;

\item the subgraph of $\Gs$ induced on $\{\mathsf v_{\pi(i_1)}, \ldots, \mathsf v_{\pi(i_{\xi})}\} \cup \{\mathsf v_{\ell_1}, \ldots, \mathsf v_{\ell_\chi}\}$
contains $\mathbf T$ as a subgraph with leaf vertices $\{\mathsf v_{\pi(i_1)}, \ldots, \mathsf v_{\pi(i_{\xi})}\}$.
\end{itemize}
(In this case, we then set $\pi(j_i) = \ell_i$ for $1\leq i\leq \chi$.)
In the above, we require the non-leaf vertices of isomorphic copies of $\mathbf T$ to be contained in ``new'' vertices in order to make sure that the edges we add per each iteration are disjoint. In order for the existence of isomorphic copies of $\mathbf T$, we need to pose some carefully chosen balanced conditions on $\mathbf T$; see Lemma~\ref{lem:designing-tree}, \eqref{eq:balanced} and \eqref{eq:balanced-2} below.

\begin{figure}[ht]
	\centering
	\includegraphics[scale=0.9]{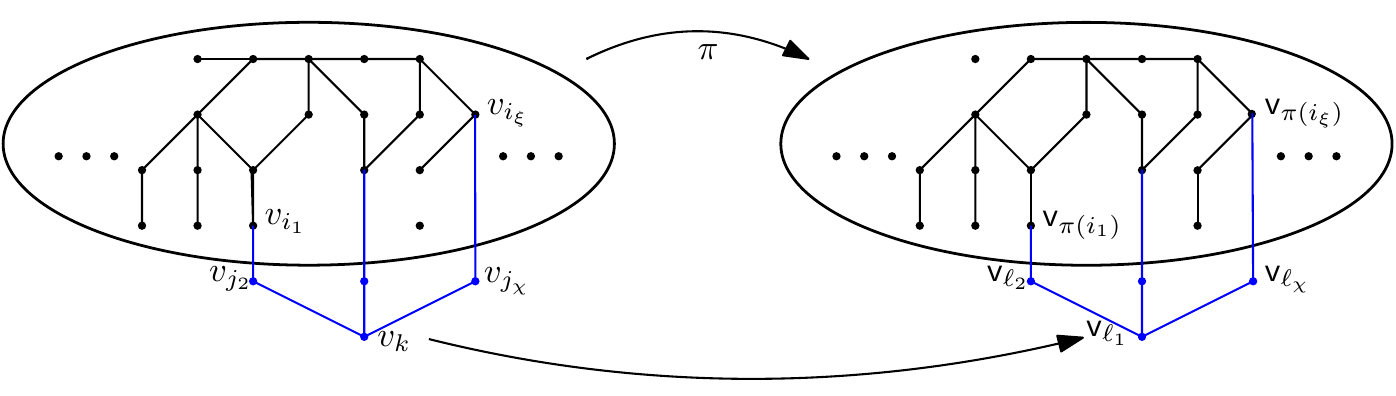}
	\caption{A single step in the iterative matching algorithm: the black vertices are matched ones, and the blue vertices are the ones matched in the current step (where a tree is added on both sides with leaves being black vertices).}
	\label{fig:matching_alg}
\end{figure}

The existence of $\mathbf T$, once appropriate balanced conditions are posed, is not that hard to show in light of \cite{BCL19, MWXY21+}. However, the major challenge comes in since we need to yet treat the correlations between different iterative steps. This is indeed fairly delicate and most of the difficult arguments in this paper (see Section~\ref{sec:complimentary-proof}) are devoted to address this challenge. 

Finally, we point out that in order for a high precision approximation of $(2\alpha - 1)^{-1}$, we will have to choose integers $\chi, \zeta$ fairly large. Since our algorithm is required to check all $\chi$-tuples over unmatched vertices, we only obtain an algorithm with polynomial running time where the power may grow in the precision of approximation. It remains an interesting question whether a polynomial time algorithm with a fixed power can find a matching that approximates the maximal overlap up to a constant that is arbitrarily close to 1.

\section{Proof of Theorem~\ref{thm:main}}
\subsection{Upper bound of the maximal overlap}\label{subsec:upper-bound}
In this subsection, we prove the upper bound of Theorem~\ref{thm:main} in the next proposition.
\begin{proposition}\label{prop-upper-bound}
	For any constant $\rho>\frac{1}{2\alpha-1}$, we have 
	\begin{equation}
		\label{eq:upper-bound}
		\Pb\left[\max_{\pi\in S_n}\mathrm{Overlap}(\pi)\ge \rho n\right]\to 0\,,\quad\text{as }n\to\infty\,.
	\end{equation}
\end{proposition}
\begin{proof}
	By a union bound we see the left hand side of \eqref{eq:upper-bound} is bounded by
	\begin{equation*}
		\sum_{\pi\in S_n}\Pb\left[\mathrm{Overlap}(\pi)\ge \rho n\right]=n!\Pb\left[\mathbf B\ge \rho n\right],
	\end{equation*}
where $\mathbf B$ is distributed as a binomial variable $\mathbf B(\binom{n}{2},p^2)$. By standard large deviation estimates for binomial variables (see e.g. \cite[Theorem 4.4]{MU05}), we have
\[
\Pb\left[\mathbf B\ge \rho n\right]\le \exp\left(-\rho n\log\left(\frac{\rho n}{\binom{n}{2}p^2}\right)+\rho n\right)=\exp\left(-\rho(2\alpha-1+o(1))n\log n+O(n)\right),
\]
which is of smaller magnitude than $\exp(-n\log n)\ll (n!)^{-1}$. This completes the proof.
\end{proof}

We introduce some asymptotic notation which will be used later. For non-negative sequences $f_n$ and $g_n$, we write $f_n \lesssim g_n$ if there exists a constant $C>0$ such that $f_n \leq C g_n$ for all $n\geq 1$. We write $f_n \asymp g_n$ if $f_n \lesssim g_n$ and $g_n \lesssim f_n$.

\subsection{Preliminaries of the algorithm}\label{subsubsec:tree}

The rest of the paper is devoted the description and analysis of the algorithm in Theorem~\ref{thm:main}. 
Since $\alpha\in (1/2,1)$, we may pick an integer $k\ge 1$ such that $\frac{1}{2\alpha-1}\in(k,k+1]$. Henceforth we will fix $\varepsilon>0$ together with a positive constant $\eta$ sufficiently close to $0$. More precisely, we choose $\eta>0$ such that
\begin{align}
	\label{eq:eta}
	&\frac{1-2\eta}{2\alpha+2\eta-1}>\frac{1-\varepsilon}{2\alpha-1}\,,\quad \frac{1}{2\alpha+2\eta-1}>k\,,\\
	\label{eq:eta-2}
	\mbox{and }&\bigg\lfloor\frac{(k+1)(2{(\alpha+\eta)}-1)-1}{2-2{(\alpha+\eta)}}\bigg\rfloor=\bigg\lfloor\frac{(k+1)(2{\alpha}-1)-1}{2-2{\alpha}}\bigg\rfloor\,,
\end{align}
where $\lfloor x\rfloor$ is the greatest integer $\le x$ and $\lceil x\rceil$ is the minimal integer $\ge x$.

As described in Section~\ref{sec:overview}, our  matching algorithm works by searching for isomorphic copies of  certain well-chosen tree in an iterative manner. We next construct this tree with a number of ``balanced conditions'' which will play an essential role in the analysis of the algorithm later.
Before doing so, we introduce some notation conventions. For any simple graph $\mathbf H$, let $V(\mathbf H),E(\mathbf H)$ be the set of vertices and edges of $\mathbf H$, respectively. Throughout the paper, we will use bold font for a graph (e.g., $\mathbf T$) to emphasize its structural information; we use normal font for  subgraphs  of $G$ (e.g., $ {T}$) and  we use mathsf font for subgraphs of $\Gs$ (e.g., $\mathsf T$). In addition, we use sans-serif font such as $\operatorname L,\operatorname Q$ to denote the subsets of $\{1,\cdots,n\}$ which serve as the subscripts of vertices in $G$ and $\Gs$.
\begin{lemma}
	\label{lem:designing-tree}
	There exist an {\emph{irrational}} number $\alpha_\eta$ and two integers $\chi,\zeta$ such that there is a rooted tree $\mathbf T$ with leaf set $\mathbf L$ and non-leaf set $\mathbf Q$ such that the following hold:\\ 
	\noindent \emph{(i)} $\alpha<\alpha_\eta<\alpha+\eta, |\mathbf Q|=\chi, |E(\mathbf T)|=\zeta$ and $0<\chi-(2\alpha_\eta-1)\zeta<1-\alpha_\eta$.\\
	\noindent \emph{(ii)} For each $u\in \mathbf Q$ which is adjacent to some leaves, $u$ is adjacent to exactly $k$ leaves.\\
	\noindent \emph{(iii)} For any subgraph $\mathbf{F}\subsetneq \mathbf T$ with $\mathbf L \subset \mathbf F$, it holds that
	\begin{equation}
		\label{eq:balanced}
		|V(\mathbf T)\setminus V(\mathbf F)|<\alpha_\eta |E(\mathbf T)\setminus E(\mathbf F)|\,.
	\end{equation}
    \noindent \emph{(iv)} For any subtree $\mathbf T_0\subsetneq \mathbf T$, it holds that
    \begin{equation}
    	\label{eq:balanced-2}
    	|V(\mathbf T_0)\cap \mathbf Q|-(2\alpha_\eta-1)|E(\mathbf T_0)|>\chi-(2\alpha_\eta-1)\zeta\,.
    \end{equation}
\end{lemma}
\begin{proof}
Recall that $k=\lceil\frac{1}{2\alpha-1}\rceil-1$. First, we claim there exist positive integers $n_1\leq n_2\leq\cdots\leq n_l$ such that
\begin{equation}
\label{eq:alpha tilde}
	2\alpha-1<\frac{1+\frac{1}{n_1}+\frac{1}{n_1n_2}+\cdots+\frac{1}{n_1n_2\cdots n_l}}{k+1+\frac{1}{n_1}+\frac{1}{n_1n_2}+\cdots+\frac{1}{n_1n_2\cdots n_{l-1}}}<2(\alpha+\eta)-1\,.
\end{equation}
In order to see this, pick an \emph{irrational} number $\beta$ such that $$ 2\alpha-1<\frac{1+\beta}{k+1+\beta}<2(\alpha+\eta)-1\,.$$
Then we can express $\beta$ as the infinite sum $$\beta=\frac{1}{n_1}+\frac{1}{n_1n_2}+\frac{1}{n_1n_2n_3}+\cdots$$
with positive integers $n_1\le n_2\le \cdots$ determined by the following inductive procedure: assuming $n_1,\dots,n_{k-1}$ have been fixed, we choose $n_k$ such that
\[
\frac{1}{n_k} < n_1\cdots n_{k-1}\left(\beta-\frac{1}{n_1}-\cdots-\frac{1}{n_1\cdots n_{k-1}}\right) < \frac{1}{n_k-1}
\]
(note that the irrationality of $\beta$ ensures the strict inequality above). It is straightforward to verify that this yields the desired expression. With this at hand, we obtain the desired relation \eqref{eq:alpha tilde} by an appropriate truncation. 

 Writing $\ell=n_1n_2\cdots n_l$, we set
 $$\chi=\ell\left(1+\frac{1}{n_1}+\cdots+\frac{1}{n_1n_2\cdots n_l}\right) \mbox{ and } \zeta=\ell\Big(k + 1+\frac{1}{n_1}+\cdots+\frac{1}{n_1n_2\cdots n_{l-1}}\Big)\,.$$
Thus, we have $\zeta=k\ell+\chi-1$. We define a rooted tree $\mathbf T$ with $(\ell+2)$ generations as follows: the $0$-th generation contains a single root; for $0\leq i\leq l-1$, each vertex in the $i$-th generation has $n_{l-i}$ children; each vertex in the $l$-th generation has  $k$ children (which are leaves). It is clear that $\mathbf{T}$ has $\chi$ non-leaf vertices and $\zeta$ edges. In addition, $\mathbf{T}$ satisfies (ii). Choose $ \widetilde{\alpha}$ such that $\chi/\zeta = 2\widetilde{\alpha}-1$. Then by \eqref{eq:alpha tilde}, we see that  (i) holds for $\alpha_\eta\in (\alpha,\widetilde{\alpha})$ which are sufficiently close to $\widetilde{\alpha}$. In the rest of the proof, we will show that for some irrational $\alpha_\eta<\widetilde\alpha$ sufficiently  close to $\widetilde{\alpha}$, (iii) and (iv) also hold. We remark that the proof is somewhat technical and it may be skipped at the first reading.

We begin with \eqref{eq:balanced}. We say $\mathbf T_0$ is an \emph{entire subtree} of $\mathbf T$ if $\mathbf{T}_0$ is a subtree of $\mathbf{T}$ such that $\mathbf T_0$ contains all neighbors of $u$ in $\mathbf T$ for any vertex $u$ with degree larger than 1 in $\mathbf T_0$. For a subgraph $\mathbf{F} \subset \mathbf T$, we consider the following procedure:
\begin{itemize}
\item remove all the vertices and edges in $\mathbf F$ from $\mathbf T$ and thus what remains is a union of vertices and edges while some edges may be incomplete since its one or both endpoints may have been removed;
\item for each aforementioned incomplete edge that remains and for each of the endpoint that has been removed, we add a \emph{distinct} vertex to replace the removed endpoint. 
\end{itemize}
\begin{figure}[ht]
\centering
\includegraphics[scale=0.8]{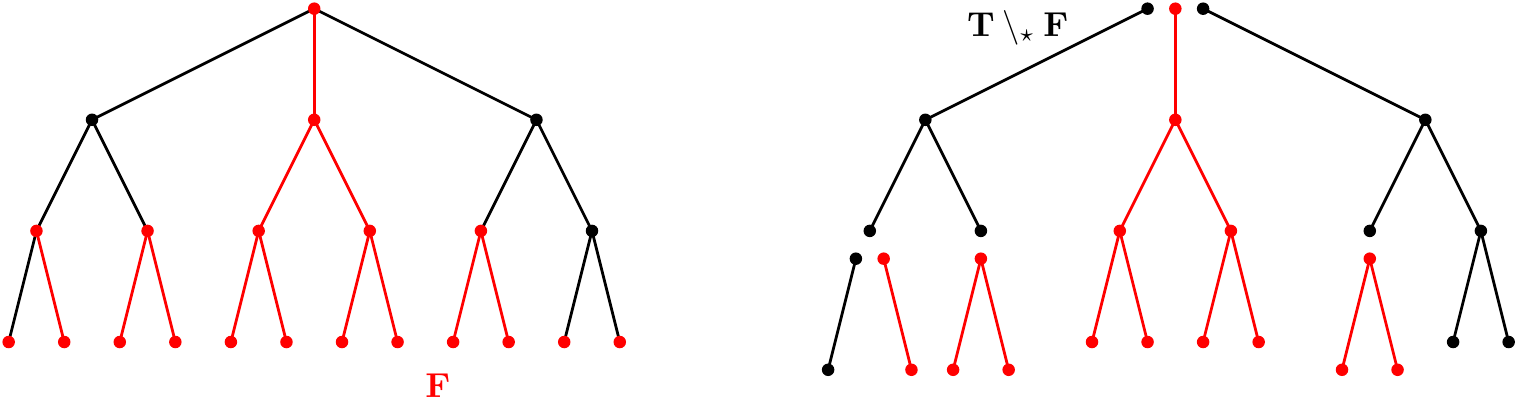}
\caption{The red part on the left is the subgraph $\mathbf F$. The black part on the  right is $\mathbf T\setminus_\star \mathbf F$. There are three entire subtrees (black) in $\mathbf T\setminus_\star \mathbf F$.}
\label{fig:Entire trees}
\end{figure}
At the end of this procedure, we obtain a collection of trees that are both edge-disjoint and vertex-disjoint and we denote this collection of trees as $\mathbf T\setminus_\star \mathbf F$. In addition, for each tree in $\mathbf T\setminus_\star \mathbf F$, it can be regarded as a subtree of $\mathbf T$ (if we identify the added vertex in the tree in $\mathbf T\setminus_\star \mathbf F$ to the corresponding vertex in $\mathbf T$) and one can verify that each such subtree is an entire subtree of $\mathbf T$. 

Therefore, in order to prove \eqref{eq:balanced} it suffices to show that for any entire subtree $\mathbf{T}_0 \subset \mathbf T$ where $\mathbf Q_0$ is the collection of vertices with degree larger than 1 in $\mathbf T_0$, we have 
\begin{equation}\label{eq-entire-tree-bound}
|\mathbf Q_0|\leq \alpha |E(\mathbf{T}_0)|
\end{equation}
 (note that $\alpha<\alpha_\eta$) and in addition the equality holds only when the tree is a singleton. For a subtree $\mathbf T_0$ of $\mathbf T$, assume its root $\mathbf v$ (i.e., the vertex in $V(\mathbf T_0)$ closest to the root of $\mathbf T$) lies in the $i$-th generation of $\mathbf T$. We define the height of $\mathbf T_0$ by
 \[
 h(\mathbf T_0)=\begin{cases}
 	l+1-i,\quad&\text{ if }|V(\mathbf T_0)|\le 2\text{ or }\mathbf v\in \mathbf Q_0\,,\\
 	l-i,\quad &\text{ if }|V(\mathbf T_0)|>2\text{ and }\mathbf v\notin \mathbf Q_0\,.
 \end{cases}
 \]

Since obviously the equality in \eqref{eq-entire-tree-bound} holds if $|V(\mathbf T_0)|\le 2$, we then deal with the other cases.
Let $m = |\{i: n_i=1\}|$. From \eqref{eq:alpha tilde} we see
\begin{equation*}
	\frac{1}{n_1}+\cdots+\frac{1}{n_1\cdots n_{l-1}}<\frac{(k+1)(2\widetilde{\alpha}-1)-1}{2-2\widetilde{\alpha}}<\frac{1}{n_1}+\cdots+\frac{1}{n_1\cdots n_l}\,.
\end{equation*}
Since $n_1=\cdots=n_m=1$ and $n_t\ge 2$ for $t>m$, we get from \eqref{eq:eta-2} that
\begin{equation}
	\label{eq:number of ni}
	m= \Big\lfloor\frac{(k+1)(2\widetilde{\alpha}-1)-1}{2-2\widetilde{\alpha}}\Big\rfloor=\Big\lfloor\frac{(k+1)(2{\alpha}-1)-1}{2-2{\alpha}}\Big\rfloor\,.
\end{equation}
We now  prove \eqref{eq-entire-tree-bound} for any entire subtree $\mathbf{T}_0$ with $h(\mathbf T_0) \leq m+1$.
\begin{itemize}
	\item If $1/2<{\alpha}\le3/4$, then we have $k\ge 2$ and $m=0$. Thus, $\mathbf T_0$ must be a star graph with $k$ or $k+1$ leaves, implying that 
	\[
	|\mathbf Q_0|-\alpha |E(\mathbf{T}_0)| \leq 1-k\alpha<0\,.
	\]
	\item If $3/4<\alpha<1$, then $k=1$ and  $\frac{2m+3}{2m+4}< \alpha\le \frac{2m+5}{2m+6}$. Thus, $\mathbf T_0$ must be a chain with length at most $m+2$, implying that
	\[
	|\mathbf Q_0|-\alpha |E(\mathbf{T}_0)|\le (m+1)-(m+2)\alpha<0\,.
	\] 
\end{itemize}

For an entire subtree $\mathbf T_0$ with $h(\mathbf T_0) \geq m+2$, we prove \eqref{eq-entire-tree-bound} by proving it for generalized entire subtrees, where $\mathbf T_0$ is a generalized entire subtree if $\mathbf T_0$ contains all the children (in $\mathbf T$) of $u$ for each $u\in \mathbf Q_0$. We prove a stronger version of \eqref{eq-entire-tree-bound} since we will prove it by induction (and it is not uncommon that proving a stronger statement by induction makes the proof easier). For the base case when $h(\mathbf T_0) = m+2$, it follows since
\[
|\mathbf Q_0|-\alpha |E(\mathbf T_0)|\le
\begin{cases}
	1-2\alpha,\quad &1/2<\alpha\le 3/4\,,\\
	(2m+3)-(2m+4)\alpha,\quad &3/4<\alpha<1\,,
\end{cases} 
\]
which is negative by preceding discussions. Now suppose that \eqref{eq-entire-tree-bound} holds whenever $\mathbf T_0$ is a generalized entire subtree with $h(\mathbf T_0) = h$, and we next consider a generalized entire subtree $\mathbf T_0$ with $h(\mathbf T_0)=h+1$. Denote the non-leaf vertex of $\mathbf T_0$ at the $(l-h-1)$-th generation (in $\mathbf T$) by $\mathbf o_0$ (this is well-defined since $\mathbf o_0$ uniquely exists as long as $\mathbf T_0$ contains more than $2$ vertices). In addition, denote  the subtrees of $\mathbf{T}_0$ rooted at children of $\mathbf o_0$ by $\mathbf{T}_1,\dots,\mathbf{T}_s$ where $\mathbf T_i$ contains all the descendants (in $\mathbf T_0$) of its root. Then $\mathbf{T}_i$ is a generalized entire subtree with height $h$ for $i = 1, \ldots, s$. Thus, by our induction hypothesis we get (denoting by $\mathbf Q_i$ the non-leaf vertices in $\mathbf T_i$)
\begin{equation}
	\label{eq:induction-hypothesis}
	|\mathbf Q_i|\le \alpha|E(\mathbf{T}_i)| \mbox{ for all } 1\le i\le s\,.
\end{equation}
Note that $|\mathbf Q_0|=1+|\mathbf Q_1|+\ldots+|\mathbf Q_s|$ and $|E(\mathbf{T}_0)|\geq s+|E(\mathbf{T}_1)|+\cdots+|E(\mathbf{T}_s)|$. Combined with \eqref{eq:induction-hypothesis} and the fact $s\ge 2$, it yields that $|\mathbf{Q}_0|-\alpha|E(\mathbf{T}_0)|<0$, completing the inductive step for the proof of \eqref{eq-entire-tree-bound} (so this proves \eqref{eq:balanced}).

We next prove \eqref{eq:balanced-2}. It suffices to show that for any subtree $\mathbf{T}_0\subsetneq \mathbf T$, 
\begin{equation}
	\label{eq:density lower bound}
	\frac{|V(\mathbf{T}_0)\cap \mathbf Q|}{|E(\mathbf{T}_0)|}>\frac{|\mathbf Q|}{|E(\mathbf{T})|}=\frac{\chi}{\zeta}=2\widetilde{\alpha}-1\,.
\end{equation}
Indeed, \eqref{eq:density lower bound} implies that 
\[
|V(\mathbf T_0)\cap \mathbf Q|-(2\widetilde\alpha-1)|E(\mathbf T_0)|>0=\chi-(2\widetilde{\alpha}-1)\zeta \mbox{ for all } \mathbf T_0\subsetneq \mathbf T\,.
\]
Writing $A$ the set of $\alpha_\eta$ for which \eqref{eq:balanced-2} holds, we get from the preceding inequality that $\widetilde{\alpha}\in A$. By the finiteness of $\mathbf T$, we have that $A$ is an open set and thus we can pick an irrational $\alpha_\eta \in A$ with $\alpha_\eta<\widetilde{\alpha}$ and sufficiently close to $\widetilde{\alpha}$.

Define
$\mathbb{T}=\{\mathbf{T}_v:  v\in \mathbf T\}$ where $\mathbf{T}_v$ is a subtree rooted at $v$ containing all descendants of $v$ in $\mathbf T$.
We first verify \eqref{eq:density lower bound} for subtrees in $\mathbb{T}$. Denoting $n_0=1$, we see that for any $0\le j\le l$ and any vertex $v$ in the $(l-j)$-th generation, 
\[
\frac{|V(\mathbf T_v)\cap \mathbf Q|}{|E(\mathbf T_v)|}=\frac{\frac{1}{n_0}+\frac{1}{n_1}+\cdots+\frac{1}{n_1\cdots n_j}}{k+\frac{1}{n_0}+\frac{1}{n_1}+\cdots+\frac{1}{n_1\cdots n_{j-1}}}\stackrel{\operatorname {def}}{=}2\alpha_j-1\,.
\]
We now show $\alpha_j>\widetilde{\alpha}$ for any $0\le j\le l-1$. This is true for $j=0$ by the choice of $\eta$ in \eqref{eq:eta}. For $j \geq 1$, we assume otherwise that the inequality does not hold for some $j\geq 1$. Then we may pick the minimal $j_0\in \{1,2,\dots,l-1\}$ such that $\alpha_{j_0}\leq \widetilde{\alpha}$.  By minimality of $j_0$, we have that $\alpha_{j_0-1}>\widetilde{\alpha}$ and thus $n_{j_0}>\frac{1}{2\widetilde{\alpha}-1}$. Since $\{n_j\}_{j=0}^{l}$ is non-decreasing, we get $n_j>\frac{1}{2\widetilde{\alpha}-1}$ for all $j\ge j_0$, and thus $\alpha_j<\widetilde{\alpha}$ for any $j>j_0$. This implies $\widetilde{\alpha}=\alpha_l<\widetilde{\alpha}$, arriving at a contradiction and thus verifying \eqref{eq:density lower bound} for subtrees in $\mathbb{T}$.

Finally, we reduce the general case to subtrees in $\mathbb{T}$ as follows: for any subtree $\mathbf T_0\subset\mathbf T$, we denote its root by $\mathbf o_0$ (this is the closest vertex in $\mathbf T_0$ to the root of $\mathbf T$) and consider $\mathbf T_{\mathbf o_0}\in\mathbb T$.  It is clear that $\mathbf T_0$ can be obtained from $\mathbf T_{\mathbf o_0}$ by deleting some vertices and edges. Further, the numbers of vertices and edges deleted are the same and let us denote by $N$ this number. We then have
\begin{equation}
	\label{eq:iteration of Tv}
	\frac{|V(\mathbf T_0)\cap \mathbf Q|}{|E(\mathbf T_0)|}\ge \frac{|V(\mathbf T_{\mathbf o_0})\cap \mathbf Q|-N}{|E(\mathbf T_{\mathbf o_0})|-N}\geq \frac{|\mathbf T_{\mathbf o_0}\cap \mathbf Q|}{E(\mathbf T_{\mathbf o_0})}\geq {2\widetilde{\alpha}-1}\,,
\end{equation}
with equality holds if and only if $N=0$ and $\mathbf T_{\mathbf o_0}=\mathbf{T}$ (that is, $\mathbf{T}_0=\mathbf{T}$). This completes the proof of \eqref{eq:density lower bound}. 
\end{proof}

In what follows, we fix $\alpha_\eta,\chi,\zeta$ and the tree $\mathbf T$ given in Lemma~\ref{lem:designing-tree}, and let $\xi = \zeta-\chi+1$. We label the vertices of $\mathbf T$ by $1,2,\dots,\chi,\chi+1,\dots,\chi+ \xi$, such that the root is labeled by $1$, $\mathbf Q=\{1,2,\dots,\chi\}$ and $\mathbf L=\{\chi+1,\dots,\chi+ \xi\}$. Recall that
\begin{equation}
	\label{eq:edge-of-the-tree}
	E(\mathbf T)=\{(i,j):1\le i<j\le \chi+ \xi, i\mbox{ is adjacent to } j \mbox{ in }\mathbf T\}\,. 
\end{equation}

For later proof, it would be convenient to consider \ER graphs with edge density $p_\eta=n^{-\alpha_\eta}$, and we can reduce our problem to this case by monotonicity. Indeed, since $\alpha_\eta>\alpha$ we have $p_\eta < p$ for all $n$ and thus $\mathbf G(n,p)$ is stochastically dominated by $\mathbf G(n,p_\eta)$. In addition, since monotonicity in $p$ naturally holds for $\mathrm{Overlap}(\pi)$, we may change the underlying distribution from $\mathbf{G}(n,p)$ to $\mathbf G(n,p_\eta)$ and prove the lower bound for the latter. In what follows we use $G$ and $\Gs$ to denote two independent \ER graphs $\mathbf G(n,p_\eta)$ (as well as the respective adjacency matrices); we drop the superscript $\eta$ here for notation convenience. We denote by $\Pb$ the joint law of $(G, \mathsf G)$. Finally, we remark that the assumption of irrationality for $\alpha_\eta$ in Lemma~\ref{lem:designing-tree} is purely technical, which will be useful in later proofs (for the purpose of ruling out equalities).

\subsection{Description of the algorithm}\label{subsubsec:algo}

We first introduce some notations.  For any nonempty subset $\operatorname S\subset [n]$ and any integer $1\le m\le |\operatorname S|$, define (recalling \eqref{eq:edge-of-the-tree})
\begin{equation}
	\label{eq:def-ordered-tuple}
	\mathfrak{A}(\operatorname S,m)=\{(i_1,\dots,i_m) \in S^m: i_1,\dots,i_m\text{ are distinct}\}\,.
\end{equation}
In addition, for $m\in \{\chi,\xi\}$, we sample a random total ordering $\prec_m$ on the set $\mathfrak{A}([n],m)$ uniformly from all possible orderings and fix it (This can be done efficiently, since it is equivalent to sampling a uniform permutation on a set with cardinality polynomial in $n$, which can be done in poly-time in $n$).
We will omit the subscript when it is clear in the context. For any tuple $\operatorname L = (t_1, \ldots, t_m)$ and any permutation $\pi$, we write $\pi(\operatorname L) = (\pi(t_1), \ldots, \pi(t_m))$. For a $\xi$-tuple $\operatorname L=(t_{\chi+1},\dots,t_{\chi+ \xi})\in \mathfrak{A}([n], \xi)$ and a $\chi$-tuple $\operatorname Q=(t_1,\dots,t_\chi)\in \mathfrak{A}([n],\chi)$ with $\operatorname L\cap \operatorname Q=\emptyset$ (i.e., the coordinates of $\operatorname L$ are disjoint from the coordinates of $\operatorname Q$), we let $  L =\{v_i:i\in \operatorname L\}$, $Q =\{v_i:i\in \operatorname Q\}$ and define
\begin{equation}\label{eq:def-TsimT'}
\{L\Join_{G}   Q\} = \{G_{t_i,t_j}= 1 \mbox{ for all } (i,j)\in E(\mathbf T)\}\,.
\end{equation}
 Let $\{L\not\Join_{G}Q\}$ be the complement of $\{L\Join_{G}  Q\}$. In addition, similar notations of $\mathsf L$, $\mathsf Q$, $\{\mathsf L\Join_{\Gs}\mathsf Q\},\{\mathsf L\not\Join_{\Gs}\mathsf Q\}$ apply for the graph $\Gs$, where the $G_{i, j}$'s in $\eqref{eq:def-TsimT'}$ are replaced by corresponding $\mathsf G_{i, j}$'s.
For any $\xi$-tuple $\operatorname L\subset \mathfrak{A}([n], \xi)$ and any subset $\operatorname U\subset [n]$, let $\{L\Join_{G} U\}$ be the event that $\{L\Join_{G} Q\}$ occurs for at least one $\chi$-tuple $\operatorname  Q\subset \mathfrak{A}(\operatorname U\setminus\operatorname L,\chi)$ where  $U = \{v_i:i\in \operatorname U\}$.  Further, for $r\in \operatorname U$, let $\{L\Join_{G,r} U\}$ denote the event that there exists $\operatorname Q = (t_1, \ldots, t_{\chi}) \in \mathfrak A(\operatorname U \setminus \operatorname L, \chi)$ such that the event
 $\{L\Join_{G}Q\}$ holds and $t_1 = r$. Similar notations of $\{\mathsf L\Join_\Gs \mathsf U\},\{\mathsf L\Join_{\Gs,r} \mathsf U\}$ apply for $\Gs$. Moreover, we define a mapping $\Pi = \Pi_\pi: [n]\to\mathsf V$ corresponding to $\pi$ such that $\Pi(i)=\mathsf v_{\pi(i)}$, and we define a mapping $I_{G}:[n]\to V$ by $I_{G}(i)= v_i$ (we also define $I_{\Gs}$ similarly).
 
Now we are ready to describe our matching algorithm. Roughly speaking, our algorithm proceeds iteratively in the following greedy sense: in the $s$-th step, assuming we have already matched a set $\operatorname M_{s-1}$ (i.e., we have determined the value of $\pi^*(i)$  for $i\in \operatorname M_{s-1}$), we then pick some $u\in \operatorname R_{s-1}=[n]\setminus \operatorname M_{s-1}$ and try to find a triple of tuples $(L,Q,\mathsf Q)$ with
\begin{equation}
	\begin{split}
		  L&=\{(v_{t_{\chi+1}},\cdots,v_{t_{\chi+ \xi}}):\,(t_{\chi+1}\cdots,t_{\chi+ \xi}) \in \mathfrak A(\operatorname M_{s-1}, \xi)\},\\
		  Q&=\{(v_{t_{1}},\cdots, v_{t_{\chi}}):\,(t_1,\dots,t_\chi)\in \mathfrak A(\operatorname R_{s-1},\chi)\},\\
		\mathsf Q&=\{(\mathsf v_{t_1'},\cdots,\mathsf v_{t_\chi'}):\,(t_1',\dots,t_\chi')\in \mathfrak A([n]\setminus \pi^*(\operatorname M_{s-1}),\chi)\}\,.
	\end{split}\notag
\end{equation}
such that $t_1=u$ and both $\{L\Join_{G}  Q\}$ and  $\{\Pi\left( I_{G}^{-1}(L)\right)\Join_{\Gs} \mathsf Q\}$ happen. If such $L,Q,\mathsf Q$ exist, we let $\pi^*(t_j)=t_j'$ and $\Pi(t_i)=\mathsf v_{t_j'}$ for $1\le j\le \chi$ (that is, we determine the values of $\pi^*$ at $t_1, \ldots, t_{\chi}$ in this step); else we just choose the value $\pi^*(u)$ arbitrarily (that is, we just determine the value of $\pi^*$ at $u$ in an arbitrary manner). This completes the construction for the $s$-th step. We expect that in most steps we are able to find such triples and thus the increment for the overlap is at least $\zeta$. Since there are around ${n}/{\chi}$ steps in total, this would imply that at the end of the procedure $\mathrm{Overlap}(\pi^*)\approx {\zeta}n/\chi \approx \frac{n}{2\alpha-1}$, as desired.

While the above algorithm may indeed achieve our goal,  we will further incorporate the following technical operations in order to facilitate the analysis of the algorithm:
\begin{itemize}
\item  Fixing a large integer $\kappa_0$ with $\kappa_0>4\zeta/\eta$, in each step we will first exclude the vertices which have been ``used'' for at least $\kappa_0$ times. 
\item In each step, when seeking the desired triple $(L,Q,\mathsf Q)$, we will check the tuples in the order given by $\prec$. 
\end{itemize}

We next present our full algorithm formally. Initially, we set $\pi^*(i)=i$ for all $1\le i\le \eta n$ and let $\operatorname M_0=\{1,2,\dots,\lfloor \eta n\rfloor\}$. Let $\operatorname R_0=[n]\setminus \operatorname M_0$. 
Set $\operatorname {EXP}_0= \operatorname {SUC}_0= \operatorname {FAIL}_0=\emptyset$. As we will see in the formal definition below, we will define $\operatorname {EXP}_s= \operatorname {SUC}_s\cup \operatorname {FAIL}_s\subset \mathfrak A([n], \xi)$ to be the set of $\xi$-tuples which are explored during the $s$-th step, where $ \operatorname {SUC}_s$ (respectively $ \operatorname {FAIL}_s$) denotes the collection of tuples which were successfully matched (respectively, failed to be matched). Our algorithm then proceeds as follows:
\begin{breakablealgorithm}\label{Alg}
			\label{algo:greedy}
			\caption{Greedy Matching Algorithm}
	\begin{algorithmic}[1]
		\STATE Define $\pi^*(i),1\le i\le \eta n$ and  $\operatorname M_0,\operatorname R_0, \operatorname {EXP}_0, \operatorname {SUC}_0, \operatorname {FAIL}_0$ as above.
		\FOR{$s=1,2,\dots$}
		\IF{$|\operatorname R_{s-1}|\ge \eta n$}
		\STATE Set $I_s=0$; set a triple $ \operatorname {MT}_s= \mathtt{null}$; set $\mathtt M_s=\operatorname M_{s-1}$.
		\FORALL{$u\in \operatorname M_{s-1}$}
		\IF{$u$ appears in at least $\kappa_0$ tuples in $\{\operatorname {MT}_t: 1\le t\le s-1\}$}
		\STATE Delete $u$ from $\mathtt M_s$.
		\ENDIF
		\ENDFOR
		\STATE Set $\operatorname {EXP}_{s}=\operatorname {SUC}_{s}=\operatorname {FAIL}_{s}=\emptyset$.
		\STATE Find the minimal element $u_s\in \operatorname R_{s-1}$.
		\STATE Let $\operatorname {CAND}_s$  be the collection of $\operatorname L\in\mathfrak A(\mathtt M_s, \xi)$ so that $\{I_{G}(\operatorname L)\Join_{G,u_s} \operatorname R_{s-1}\}$ holds.
		\STATE Label elements in $\operatorname {CAND}_s$ by $\operatorname L_1,\dots,\operatorname L_l$ such that $\operatorname L_1\prec \cdots\prec \operatorname L_l$.
		\STATE Label all  $\chi$-tuples in $\mathfrak A([n]\setminus \pi^*(\operatorname M_{s-1}),\chi)$ by $\operatorname Q_1,\dots,\operatorname Q_m$ so that $\operatorname Q_1\prec\cdots\prec \operatorname Q_m$.
		\FOR{$i=1,2,\dots,l$}
		\FOR{$j=1,2,\dots,m$}
		\STATE Check whether $\{\Pi(\operatorname L_i)\Join_{\Gs} I_\Gs(\operatorname Q_j)\}$ happens where $\Pi = \Pi_{\pi^*}$.
		\IF{$\{\Pi(\operatorname L_i)\Join_{\Gs} I_\Gs(\operatorname Q_j)\}$ happens}
		\STATE  Find a $\chi$-tuple $\operatorname Q\in \mathfrak A(\operatorname R_{s-1},\chi)$ such that $\{I_{G}(\operatorname L_i)\Join_{G, u_s} I_{G}(\operatorname Q)\}$ holds. \COMMENT{The existence of $\operatorname Q$ is guaranteed by definition of $\operatorname L_i$.}
		\STATE Set $I_s=1$ and $\operatorname {MT}_s=(\operatorname L_i,\operatorname Q,\operatorname Q_j)$.
		\STATE Add $\operatorname L_i$ into $\operatorname {SUC}_{s}$. Then break the \textbf{for} cycle, and break the \textbf{for} cycle again.
		\ELSIF{$j=m$} 
		\STATE Add $\operatorname L_i$ into $\operatorname {FAIL}_{s}$. \COMMENT{This means we have checked all possible tuples $\operatorname Q_j,j=1,2,\dots,m$ but failed to find the desired one.}
		\ENDIF
		\ENDFOR
		\ENDFOR
		\STATE Set $\operatorname {EXP}_{s}=\operatorname {SUC}_{s}\cup\operatorname {FAIL}_{s}$.
		\IF{$I_s=1$}
		\STATE  Recall $\operatorname {MT}_s=(\operatorname L_i,\operatorname Q, \operatorname Q_j)$. Set $\pi^*$ on coordinates of $\operatorname Q$ so that $\pi^*(\operatorname Q) = \operatorname Q_j$.
		\STATE Set $\operatorname M_{s}$ as the union of $\operatorname M_{s-1}$ and all coordinates in $\operatorname Q$; set $\operatorname R_{s}=[n]\setminus \operatorname M_{s}$.
		\ELSE 
		\STATE Set $\pi^*(u_s)$ as the minimal element in $[n]\setminus\pi^*(\operatorname M_{s-1})$.
		\STATE Set $\operatorname M_{s}=\operatorname M_{s-1}\cup \{u_s\},\operatorname R_{s}=[n]\setminus \operatorname M_{s}$.
		\ENDIF
		\ELSE
		\STATE set $\pi^*(u)$ for $u\in \operatorname R_s$ arbitrarily such that $\pi^*$ becomes a permutation on $[n]$.
		\STATE Break the \textbf{for} cycle.
		\ENDIF
		\ENDFOR 
		\RETURN $\pi^*$.
	\end{algorithmic}
\end{breakablealgorithm}
\begin{remark}
	$\operatorname{CAND}_s$ stands for the candidate tuples which may be successfully matched in the $s$-th step, while $\operatorname{EXP}_s$ denotes for the set of tuples in $\operatorname{CAND}_s$ that have been checked during the algorithm.  In most part of this paper, we do not distinguish $\operatorname{CAND}_s$ and $\operatorname{EXP}_s$, and they are indeed equal when $I_s=0$ (i.e., we have checked all candidate tuples but failed to match up). However, one should keep in mind that typically it holds $|\operatorname{EXP}_{s}|\ll |\operatorname{CAND}_s|$. Heuristically, this is because the checking procedure is according to a uniform ordering $\prec$ and it stops as long as any successful matching triple is found. We shall make this precise and incorporate it as an ingredient for proofs in Section~\ref{sec:typical-good-event}. 
\end{remark}

\subsection{Aanalysis of the algorithm}\label{subsec:algo-analysis}
In this subsection we prove that Algorithm~\ref{algo:greedy} satisfies the condition of Theorem~\ref{thm:main}. To this end, we need to analyze the conditional probability for the event $\{I_s=1\}$ given the behavior of Algorithm~\ref{algo:greedy} in previous steps. For notation convenience, in what follows we write $\pi = \pi^*$ and write $\operatorname R'_s = [n] \setminus \pi(\operatorname M_s)$. Since in each step we always have $|\operatorname R_{s+1}|\geq |\operatorname R_s|-\chi$, the algorithm runs for at least $S=\lfloor \frac{1-2\eta}{\chi}n\rfloor$ steps. We will only consider the first $S$ steps. For each $1\le s\le S$, define $\mathcal{F}_{s-1}$ to be the $\sigma$-field generated by $\operatorname M_t, \operatorname {SUC}_t, \operatorname {FAIL}_t, \operatorname {CAND}_t, I_t,  \operatorname {MT}_t$ for $t = 1, \ldots, s-1$ as well as $\pi(i)$ for $i\in\operatorname M_{s-1}$ (here we denote the matching triples in $\operatorname {MT}_t$ by $(\operatorname L_t,\operatorname Q_t,\operatorname Q'_t)$ if $I_t=1$, and denote $\operatorname {MT}_t=\mathtt{null}$ if $I_t=0$). 
Then $\mathcal{F}_{s-1}$ contains all the information generated by Algorithm~\ref{algo:greedy} in the first $s-1$ steps. Thus conditioning on a realization of $\mathcal{F}_{s-1}$ is equivalent to conditioning on a realization of the first $s-1$ steps of Algorithm~\ref{algo:greedy}. We further denote $\mathcal{F}_{s-1/2}$ as the $\sigma$-field generated by $\mathcal{F}_{s-1}$ and $ \operatorname{CAND}_{s}$.

With slight abuse of notation, we will write $\Pb[\cdot\mid \mathcal{F}_{s-1}]$ (respectively $\Pb[\cdot\mid \mathcal{F}_{s-1/2}]$) for the conditional probability given some particular realization of $\mathcal{F}_{s-1}$ (respectively some realization of $\mathcal{F}_{s-1/2}$). Let
\begin{equation}
		\label{eq:Fail_s-1}
\begin{aligned}
	\mathsf{Fail}_{s-1}=& \bigcup_{1\le t\le s-1}\Big\{\big(\Pi(\operatorname L),I_\Gs(\operatorname Q')\big):\operatorname L\in\operatorname {FAIL}_{t},\,\operatorname Q'\in \mathfrak{A}(\operatorname R_t',\chi)\Big\}\\
\bigcup &\bigcup_{1\le t\le s-1,\atop I_t=1}\Big\{\big(\Pi (\operatorname L_{t}),I_\Gs(\operatorname Q')):\operatorname Q'\in \mathfrak{A}(\operatorname R_t',\chi),\operatorname Q'\prec \operatorname Q'_t\Big\}
\end{aligned}
\end{equation}
and
\begin{equation}
	\label{eq:Suc_s-1}\text{Suc}_{s-1}=\bigcup_{1\le t\le s-1,\atop I_t=1}\{(I_G(\operatorname{L}_t),I_G(\operatorname{Q}_t))\}\,,\quad
	\mathsf{Suc}_{s-1}= \bigcup_{1\le t\le s-1,\atop I_t=1}\{(\Pi(\operatorname{L}_t),I_{\mathsf G}(\operatorname{Q}_t'))\}\,.
\end{equation}
Since the two graphs are independent, we have that for any event $\mathcal B$ measurable with respect to  $\mathsf G$,
$$
\Pb[\mathcal B\mid \mathcal{F}_{s-1/2}]=\Pb[\mathcal B\mid \mathcal{F}_{s-1}]=\Pb[\mathcal B\mid \mathcal A_s^1,\mathcal A_s^2]
$$
with
\begin{equation}
	\label{eq:def-of-As}
	\mathcal A_s^1=\bigcap_{(\mathsf L,\mathsf Q)\in \mathsf{Fail}_{s-1}}\{\mathsf L\not\Join_\Gs \mathsf Q\}\,,\quad \mathcal A_s^2=\bigcap_{(\mathsf L,\mathsf Q)\in \mathsf{Suc}_{s-1}}\{\mathsf L\Join_{\Gs} \mathsf Q\}\,.
\end{equation}

Now we investigate the event $\{I_s=0\}$ conditioned on some realization of $\mathcal{F}_{s-1}$.
As in the algorithm, we first pick $u_s\in \operatorname R_{s-1}$ and find all $\xi$-tuples $\operatorname{L}\in  \operatorname {CAND}_s$ (recall that $I_G(\operatorname{L})\Join_{G,u_s} R_{s-1}$ for $\operatorname{L}\in  \operatorname {CAND}_s$). Note that whenever $\mathcal{F}_{s-1}$ is given, this procedure is measurable with respect to $G$. Provided with $\operatorname {CAND}_s$, we have that $I_s=0$ is equivalent to
\[
\sum_{\operatorname L \in  \operatorname {CAND}_s} \mathbf{1}_{\{\Pi(\operatorname L)\Join_\Gs \mathsf R_{s-1}\}}=0 \mbox{ where } \mathsf R_{s-1}=I_\Gs(\operatorname R_{s-1}')\,.
\]
Since our aim is to lower-bound the preceding sum, it turns out more convenient to consider the sum over tuples with additional properties. To this end, we will define certain $s$-good tuples ($s$-good is measurable with respect to $\mathcal F_{s-1}$; see Definition~\ref{def:s-good} below) and consider
\begin{equation}
	\label{def:X_s}
	X_s\stackrel{\operatorname {def}}{=}\sum_{\operatorname L \in  \operatorname {CAND}_s\atop \operatorname{L}\text{ is }s\text{-good}} \mathbf{1}_{\{\Pi(\operatorname L)\Join_\Gs \mathsf R_{s-1}\}}\,.
\end{equation}
Clearly, $\{I_s=0\} \subset \{X_s = 0\}$. 

For any $1\le s\le S$, we shall define two good events $\mathcal G_s^i,i=1,2$, where $\mathcal G_s^1$ is measurable with respect to $\mathcal{F}_{s-1}$ and $\Gc_s^2$ is measurable with respect to $\mathcal{F}_{s-1/2}$. The precise definitions of $\Gc_s^1,\Gc_s^2$ will be given in the next section. Roughly speaking, $\mathcal G_s^1$ states that the cardinality of $\operatorname {Fail}_{s-1}$ is not unusually large, while $\Gc_s^2$ says that the number of $s$-good tuples in $\operatorname {CAND}_s$ is not too small and the intersecting profile for pairs of $s$-good tuples in $\operatorname {CAND}_s$ behaves in a typical way. 
 
The following three propositions, whose proofs are postponed until Section~\ref{sec:complimentary-proof}, are key ingredients for the proof of Theorem~\ref{thm:main}.
\begin{proposition}
	\label{prop:one-point-est}
	For  $1\le s\le S$, any realization of $\mathcal{F}_{s-1}$ and any $\xi$-tuple $\operatorname L\in \mathfrak{A}(\operatorname{M}_{s-1},\xi)$,
	\begin{equation}
		\label{eq:one-point-upper-bound}
		\Pb[\Pi(\operatorname L)\Join_\Gs \mathsf R_{s-1}\mid \mathcal A_s^1,\mathcal A_s^2]\le \Pb[\Pi(\operatorname L)\Join_\Gs \mathsf R_{s-1}]\,.
	\end{equation}
	Furthermore, for any realization on the good event $\Gc_s^1$ and $s$-good $\xi$-tuple $\operatorname{L}\in \mathfrak{A}(\operatorname{M}_{s-1},\xi)$, it holds uniformly that
	\begin{equation}
		\label{eq:one-point-lower-bound}
			\Pb[\Pi(\operatorname L)\Join_\Gs \mathsf R_{s-1}\mid \mathcal A_s^1,\mathcal A_s^2]\ge\big[1-o(1)\big] \Pb[\Pi(\operatorname L)\Join_\Gs \mathsf R_{s-1}]\,.
	\end{equation}
	In addition, under such additional assumptions, for some positive constants $c_1, c_2$ depending only on $\eta, \alpha_\eta$ and $\mathbf T$, we have 
	\begin{equation}
		\label{eq:one-point-est}
		c_1n^\chi p_\eta^\zeta\le\Pb[\Pi(\operatorname L)\Join_\Gs \mathsf R_{s-1}\mid \mathcal A_s^1,\mathcal A_s^2]\le c_2n^\chi p_\eta^\zeta\,.
	\end{equation}
\end{proposition}
\begin{proposition}
	\label{prop:two-point-est}
	For any realization of $\mathcal{F}_{s-1/2}$ on $\Gc_s^1 \cap \Gc_s^2$, it holds uniformly for $1\le s \le S$ that 
	\begin{equation}
		\label{eq:second-moment-bound}
		\begin{aligned}
			\mathbb{E}\left[X_s^2\mid \mathcal{F}_{s-1/2}\right]\le \big[1+o(1)\big]\big(\mathbb{E}\left[X_s\mid \mathcal{F}_{s-1/2}\right]\big)^2\,.
		\end{aligned}
	\end{equation}
\end{proposition}
\begin{proposition}
	\label{prop:good-events-are-typical}
	The good events are typical, i.e. for $1\le s\le S$, it holds uniformly that 
	\begin{equation}
		\label{eq:second-moment-est}
			\Pb\big[\Gc_s^1\big]=1-o(1)\,,\quad \Pb\big[\Gc_s^2\big]=1-o(1)\,. 
	\end{equation}
\end{proposition}
\begin{proof}[Proof of Theorem~\ref{thm:main}]
 It is clear that this algorithm runs in polynomial time, so it remains to show \eqref{eq:approximation}. For each $1\le s\le S$, whenever the realization of $\mathcal{F}_{s-1/2}$ satisfies $\Gc_s^1\cap\Gc_s^2$, we can apply the second moment method and obtain that
	\begin{align*}
	    \Pb[I_s=1\mid \mathcal{F}_{s-1/2}]\ge\Pb[X_s>0\mid \mathcal{F}_{s-1/2}]\ge \frac{\left(\mathbb{E}[X_s\mid \mathcal{F}_{s-1/2}]\right)^2}{\mathbb{E}[X_s^2\mid\mathcal{F}_{s-1/2}]}\,.
	    \end{align*}
Thus, from Proposition~\ref{prop:two-point-est} we get $\Pb[I_s=0\mid \mathcal{F}_{s-1/2}]=o(1)$, uniformly for all $1\le s\le S$ and all realization of $\mathcal{F}_{s-1/2}$ satisfying $\Gc_s^1\cap \Gc_s^2$. Combined with Proposition~\ref{prop:good-events-are-typical}, it yields that
\begin{equation}
	\label{eq:sum-is-o(n)}
\begin{aligned}
	&\ \mathbb{E}\Big[\sum_{s=1}^S\mathbf{1}_{\{I_s=0\}}\Big] = \sum_{s=1}^S\Big(\mathbb{E}\big[\mathbf{1}_{\{I_s=0\}\cap (\Gc_s^1\cap \Gc_s^2)^c}\big]+\mathbb{E}\big[\mathbf{1}_{\{I_s=0\}\cap\Gc_s^1\cap\Gc_s^2}\big]\Big)\\
	\le &\ \sum_{s=1}^S\Big(\Pb\big[(\Gc_s^1\cap \Gc_s^2)^c\big]+\mathbb{E}\big[\Pb[I_s=0\mid \mathcal{F}_{s-1/2}]\mathbf{1}_{\Gc_s^1\cap \Gc_s^2}\big]\Big)=o(n)\,.
\end{aligned}
\end{equation}
Define the increment of the target quantity $\mathrm{Overlap}(\pi)$ in the $s$-th step by
\begin{equation*}
	\mathcal O_s=\sum_{i<j, (i,j)\in \mathfrak{A}(\operatorname M_{s+1},2)\setminus \mathfrak{A}(\operatorname M_s,2)} G_{i,j}\mathsf G_{\pi(i),\pi(j)}\,.
\end{equation*}
 Then  $\{I_s= 1\}  \subset \{\mathcal O_s\ge \zeta\}$ for all  $1\le i\le S$. Therefore,
\begin{align*}
\mathrm{Overlap}(\pi)\ge&\ \sum_{s=1}^S \mathcal O_s\ge \zeta\sum_{s=1}^S \mathbf{1}_{\{I_s=1\}}=\zeta S-\zeta\sum_{s=1}^S\mathbf{1}_{\{I_s=0\}}\\
\ge&\ \frac{1-2\eta-o(1)}{2\alpha_\eta-1}n-\zeta\sum_{s=1}^S\mathbf{1}_{\{I_s=0\}}\,.
\end{align*}
Since $\frac{1-2\eta}{2\alpha_\eta-1}>\frac{1-\varepsilon}{2\alpha-1}$ by the choice of $\eta$ in \eqref{eq:eta} and Lemma~\ref{lem:designing-tree} (i), we conclude from Markov's inequality that 
\begin{equation*}
	\begin{aligned}
		\Pb\left[\mathrm{Overlap}(\pi)\le \frac{1-\varepsilon}{2\alpha-1}n\right]\le \frac{\zeta\cdot \mathbb{E}\left[\sum_{i=1}^S\mathbf{1}_{\{I_s=0\}}\right]}{\left(\frac{1-2\eta-o(1)}{2\alpha_\eta-1}-\frac{1-\varepsilon}{2\alpha-1}\right)n}\stackrel{\eqref{eq:sum-is-o(n)}}{=}o(1)\,,
	\end{aligned}
\end{equation*}
completing the proof of the theorem.
\end{proof}

\section{Complimentary proofs}\label{sec:complimentary-proof}

In this section, we prove Propositions~\ref{prop:one-point-est}, \ref{prop:two-point-est} and \ref{prop:good-events-are-typical}.
We need some more definitions and notations. For a subgraph $\mathbf H \subset \mathbf T$, we define its capacity by $$ \operatorname{Cap}(\mathbf H)=|V(\mathbf H)\cap \mathbf Q|-\alpha_\eta|E(\mathbf H)|\,.$$
For a subtree $\mathbf T_0\subset \mathbf T$, we say it is \emph{dense} if 
$\operatorname {Cap}(\mathbf F)> \operatorname {Cap}(\mathbf T)$ for any subgraph $\mathbf F\subsetneq \mathbf T_0$ with $V(\mathbf T_0)\cap \mathbf L \subset \mathbf F$.
In addition,  we
choose a subgrah $\mathbf F_0\subset \mathbf{T_0}$ with $V(\mathbf T_0)\cap \mathbf L\subset V(\mathbf{F}_0)$ such that $\operatorname {Cap}(\mathbf{F}_0)$ is minimal (among all such choices for $\mathbf F_0$). 
We define a quantity $\mathbb D(\mathbf T_0)$ by 
\begin{equation}
	\label{eq:def-of-M(T)}
\log_n \mathbb D(\mathbf T_0)=\operatorname{Cap}(\mathbf T_0)-\operatorname {Cap}(\mathbf F_0)\,.
\end{equation}

Throughout this section, we will frequently deal with the conditional probability that $\Gs$ contains a subgraph $\mathsf{T}\cong \mathbf{T}$ with fixed leaves given the existence of certain subgraph $\mathsf H$ in $\Gs$. To this end, we decompose such event according to the intersecting pattern of $\mathsf T\cap \mathsf H$.
For each possible realization $\mathsf{F}$ of $\mathsf{T}\cap \mathsf{H}$, take $\mathbf{F}\subset \mathbf T$ such that $\mathbf{F}$ is the image of 
$\mathsf{F}$ under the isomorphism from $\mathsf T$ to $\mathbf T$. Then it is straightforward from Markov's inequality that
\begin{equation}
	\label{eq:observation}
	\Pb[\exists \mathsf T\subset \mathsf G \mbox{ such that } \mathsf T \cong \mathbf T \mbox{ and }\mathsf{T}\cap \mathsf{H}=\mathsf{F}\mid \mathsf{H}\subset \Gs]\le n^{|V(\textbf{T})\setminus V(\textbf{F})|}p_\eta^{|E(\textbf{T})\setminus E(\textbf{F})|}\,.
\end{equation}
(In the above, $\mathsf H$ is a fixed subgraph with vertex set in $\mathsf V$, and the event $\mathsf H\subset \Gs$ means that every edge in $\mathsf H$ is contained in $\Gs$.)
Thus, in order to upper-bound $\Pb[\exists \mathsf T\subset \mathsf G \mbox{ such that } \mathsf T \cong \mathbf T \mid \mathsf{H}\subset \Gs]$, we may sum over the right hand side of \eqref{eq:observation} over all possible $\mathsf F$.

Finally, we introduce some notations. For $1\le s\le S$, fix a realization of $\mathcal F_{s-1}$. Recall the definitions of $\operatorname{Fail}_{s-1}$, $\operatorname{Suc}_{s-1}$, $\mathsf{Suc}_{s-1}$ and $\mathcal A_s^1,\mathcal A_s^2$ in \eqref{eq:Fail_s-1}, \eqref{eq:Suc_s-1} and \eqref{eq:def-of-As}. Clearly $\mathcal A_s^1$ is decreasing and $\mathcal A_s^2$ is increasing. We let (below $L\cup Q$ means the set of all vertices appearing in $L$ and $Q$)
	\[
	O_{s-1}= \bigcup_{(L,Q)\in \operatorname{Suc}_{s-1}}\{v:v\in L\cup Q\}\text{\quad and\quad } \mathsf O_{s-1}= \bigcup_{(\mathsf L,\mathsf Q)\in \mathsf{Suc}_{s-1}}\{\mathsf v:\mathsf v\in \mathsf L\cup \mathsf Q\}.
	\]
	Then $\mathsf O_{s-1}$ is the set of vertices involved in the event $\mathcal A_s^2$. From the algorithm we see
	\begin{equation}
		\label{eq:O_s-location}
		O_{s-1}\subset M_{s-1}=I_G(\operatorname{M}_{s-1})\,,\quad \mathsf O_{s-1}\subset \mathsf{M}_{s-1}=\Pi(\operatorname{M}_{s-1})\,.
	\end{equation}
	We define ${GO}_{s-1}$ (respectively $\mathsf {GO}_{s-1}$) as the graph on $M_{s-1}$ (respectively $\mathsf M_{s-1}$) which is the union of the trees
	that certify the events $\{L\Join_{G} Q\}$ for $(L,Q)\in \operatorname{Suc}_{s-1}$ (respectively $\{\mathsf L\Join_{\Gs} \mathsf Q\}$ for $(\mathsf L,\mathsf Q)\in \mathsf{Suc}_{s-1}$). Then it is clear that the event $\mathcal A_s^2$ is equivalent to $\mathsf{GO}_{s-1}\subset \Gs$. For a given realization of $\mathcal{F}_{s-1}$, we see that $O_{s-1},\mathsf O_{s-1}, {GO}_{s-1}$ and $\mathsf{GO}_{s-1}$ are deterministic, and in addition ${GO}_{s-1}\cong \mathsf {GO}_{s-1}$ through $\Pi\circ I_G^{-1}$ restricted on $M_{s-1}$.
 In addition, the rule of Algorithm~\ref{algo:greedy} implies 
\begin{equation}\label{eq-degree-bound}
\Delta( {GO}_s)\le \kappa_0+\Delta(\textbf{T})\,,
\end{equation}
where $\Delta(\cdot)$ is the maximal vertex degree. This is because except for the first time when a vertex is ``used'' it will always participate as a leaf vertex and thus only contributes 1 to the degree. For simplicity, we will denote $\widehat\Pb$ for the conditional probability on $\Gs$ given by $\Pb\big[\cdot|\mathcal A_s^2\big]=\Pb\big[\cdot\mid \mathsf{GO}_s\subset \Gs\big].$

\subsection{Proof of Proposition~\ref{prop:one-point-est}}

We continue to fix a realization of $\mathcal{F}_{s-1}$ as above and also fix a $\xi$-tuple $\operatorname{L}\in \mathfrak{A}(\operatorname{M}_{s-1},\xi)$. We first show the upper bound \eqref{eq:one-point-upper-bound}. Note that $\widehat{\Pb}$ is a product measure which admits the FKG inequality. As a result,
\begin{equation*}
	\label{eq:using-of-FKG}
	\begin{split}
		&\ \Pb\left[\Pi(\operatorname L)\Join_\Gs \mathsf R_{s-1}|\mathcal A_s^1,\mathcal A_s^2\right]=\widehat \Pb\left[\Pi(\operatorname L)\Join_\Gs \mathsf R_{s-1}|\mathcal A_s^1\right]\\
		\le&\ \widehat \Pb[\Pi(\operatorname L)\Join_\Gs \mathsf R_{s-1}]=\Pb[\Pi(\operatorname L)\Join_\Gs \mathsf R_{s-1}]\,,
	\end{split}
\end{equation*}
where the inequality holds because $\{\Pi(\operatorname L)\Join_\Gs \mathsf R_{s-1}\}$ is increasing and $\mathcal A_s^1$ is decreasing, and the last equality follows from independence (recall \eqref{eq:O_s-location}). This gives the upper bound.

Now we turn to the lower bound \eqref{eq:one-point-lower-bound}. The precise definitions of $\Gc_s^1$ and $s$-good tuples will be given later in this subsection, and we just assume $\mathcal{F}_{s-1}$ satisfies $\Gc_s^1$ and $\operatorname{L}$ is $s$-good for now. Let $\mathcal U$ be the event that there is a unique $\chi$-tuple  $\operatorname Q\in \mathfrak{A}(\operatorname R_{s-1}',\chi)$ such that $\Pi(\operatorname L)\Join_{\Gs} I_\Gs(\operatorname Q)$. Then we have
\begin{align}\label{eq-join-remove-A-1-2-lower}
		 &\ \Pb[\Pi(\operatorname L)\Join_\Gs \mathsf R_{s-1}\mid \mathcal A_s^1,\mathcal A_s^2]
		\ge\Pb\left[\Pi(\operatorname L)\Join_\Gs \mathsf R_{s-1},\mathcal U \mid \mathcal A_s^1,\mathcal A_s^2\right]\nonumber\\
		=&\ \frac{1}{\widehat{\Pb}[\mathcal A_s^1]}\sum_{\operatorname Q\in\mathfrak{A}(\operatorname R_{s-1}',\chi)}{\widehat\Pb\left[\Pi(\operatorname L)\Join_{\Gs}I_\Gs(\operatorname Q),\mathcal U,\mathcal A_s^1\right]}\,.
\end{align}
For each term in the preceding sum, note that $\mathcal U$ is decreasing conditioned on $\Pi(\operatorname L)\Join_{\Gs} I_\Gs(\operatorname Q)$ and thus we have
\begin{align}\label{eq-join-lower-bound-product}
		{\widehat \Pb\left[\Pi(\operatorname L)\Join_\Gs I_\Gs(\operatorname Q),\mathcal U,\mathcal A_s^1\right]}= &\  \widehat{\Pb}\left[\Pi(\operatorname L)\Join_\Gs I_\Gs(\operatorname Q),\mathcal U\right]\cdot{\widehat \Pb\left[\mathcal A_s^1|\  \Pi(\operatorname L)\Join_\Gs I_\Gs(\operatorname Q),\mathcal U\right]}\nonumber\\
		=&\ {\Pb}\left[\Pi(\operatorname L)\Join_\Gs I_\Gs(\operatorname Q),\mathcal U\right]\cdot \widehat\Pb\left[\mathcal A_s^1\mid \Pi(\operatorname L)\Join_\Gs I_\Gs(\operatorname Q),\mathcal U\right]\nonumber\\
		\ge&\ {\Pb}\left[\Pi(\operatorname L)\Join_\Gs I_\Gs(\operatorname Q),\mathcal U\right] \cdot\widehat\Pb\left[\mathcal A_s^1\mid \Pi(\operatorname L)\Join_\Gs I_\Gs(\operatorname Q)\right],
\end{align}
where the second equality follows from independence and the last inequality follows from FKG inequality ($\widehat{\Pb}[\cdot\mid \Pi(\operatorname L)\Join_\Gs I_\Gs(\operatorname Q)]$ is still a product measure).
Thus to prove the \eqref{eq:one-point-lower-bound}, the key is to show the following two lemmas.
\begin{lemma}
	\label{lem:unique-is-whp}
	For any $\operatorname L\in \mathfrak{A}(\operatorname M_{s-1}, \xi)$, it holds that
	\begin{equation}
		\label{eq: 4 goal first moment}
		\Pb\left[\Pi(\operatorname L)\Join_\Gs I_\Gs(\operatorname Q),\mathcal U\right]\ge \big[1-o(1)\big]\Pb\left[\Pi(\operatorname L)\Join_\Gs I_\Gs(\operatorname Q)\right].
	\end{equation}
\end{lemma}
\begin{lemma}
	\label{lem:structure}
	On some good event $\Gc_s^1$ (defined in Definition~\ref{def:G-s-1} below), we have that for any $s$-good $\operatorname L\in\mathfrak{A}(\operatorname M_{s-1}, \xi)$, 
\begin{equation}
	\label{eq: 5 goal first moment}
	{\widehat\Pb\left[\mathcal A_s^1\mid \Pi(\operatorname L)\Join_\Gs I_\Gs(\operatorname Q)\right]}\ge \big[1-o(1)\big]\widehat{\Pb}\left[\mathcal A_s^1\right].
\end{equation}
\end{lemma}
We now continue to complete the proof for \eqref{eq:one-point-lower-bound}. Note that
\[
{\Pb}\left[\Pi(\operatorname L)\Join_\Gs \mathsf R_{s-1}\right]\le \sum_{ \operatorname Q\in \mathfrak{A}(\operatorname R_{s-1}',\chi)}\Pb\left[\Pi(\operatorname L)\Join_\Gs I_\Gs(\operatorname Q)\right].
\]
Combined with \eqref{eq-join-remove-A-1-2-lower}, \eqref{eq-join-lower-bound-product} and Lemmas~\ref{lem:unique-is-whp} and \ref{lem:structure}, this verifies \eqref{eq:one-point-lower-bound}.

We next prove \eqref{eq:one-point-est}. The upper bound follows easily from Markov's inequality. As for the lower bound, by Lemma~\ref{lem:unique-is-whp} we have
\begin{align}
&\Pb[\Pi(\operatorname L)\Join_\Gs \mathsf R_{s-1}]\ge \sum_{\operatorname Q\in \mathfrak{A}(\operatorname R_{s-1}',\chi)}\Pb\left[\Pi(\operatorname L)\Join_{\Gs} I_\Gs(\operatorname Q),\mathcal U\right] \nonumber\\
\ge& \big[1-o(1)\big]\sum_{\operatorname Q\in \mathfrak{A}(\operatorname R_{s-1}',\chi)}\Pb\left[\Pi(\operatorname L)\Join_{\Gs} I_\Gs(\operatorname Q)\right]\gtrsim n^\chi p_\eta^\zeta\,. \label{eq-to-cite-prop-2.3-proof}
\end{align}
Combined with \eqref{eq:one-point-upper-bound} and \eqref{eq:one-point-lower-bound}, it yields \eqref{eq:one-point-est}, completing the proof of Proposition~\ref{prop:one-point-est}. 

It remains to prove Lemmas~\ref{lem:unique-is-whp} and \ref{lem:structure}.

\begin{proof}[Proof of Lemma~\ref{lem:unique-is-whp}]
	\eqref{eq: 4 goal first moment} is equivalent to $ \Pb\left[\mathcal U\mid \Pi(\operatorname L)\Join_\Gs I_\Gs(\operatorname Q)\right]\ge 1-o(1)$. We now show $\Pb\left[\mathcal U^c\mid \Pi(\operatorname L)\Join_\Gs I_\Gs(\operatorname Q)\right]=o(1)$ by a union bound. Let $\mathsf{T}$ denote the tree in $\Gs$ that certifies $\Pi(\operatorname L)\Join_\Gs I_\Gs(\operatorname Q)$ (so in particular $V(\mathsf T)=\{\mathsf v: \mathsf v\in\Pi(\operatorname L)\cup I_\Gs(\operatorname Q)\}$ and $\mathsf L=\Pi(\operatorname L)$ is the leaf set of $\mathsf T$). On the event $\mathcal U^c$, there exists another tree $\mathsf{T}'\cong \mathsf T$  with $V(\mathsf T) \neq V(\mathsf T')$ such that $\mathsf T' \subset \Gs$ and the leaf set of $\mathsf T'$ is also $\mathsf L$. Therefore, $\mathsf T$ and $\mathsf{T}'$ intersect at some subgraph of $\mathsf{F}$ with $\mathsf{L}\subset \mathsf{F}\subsetneq \mathsf{T}$.  By a union bound, we see
    \begin{align}
    	&\ \Pb[\mathcal U^c\mid \Pi(\operatorname L)\Join_\Gs I_\Gs(\operatorname Q)] \nonumber \\
    	\leq & \sum_{\mathsf{L}\subset \mathsf{F}\subsetneq\mathsf{T}}\Pb[\mbox{there exists } \mathsf{T}'\subset \Gs \mbox{ such that } \mathsf T' \cong \mathsf T \mbox{ and } \mathsf{T}'\cap \mathsf{T}=\mathsf{F}\mid \mathsf{T}\subset \Gs] \nonumber \\
    	\stackrel{\eqref{eq:observation}}{\le}&\sum_{ \mathbf L\subset\mathbf F\subsetneq \mathbf T}n^{|V(\textbf{T})\setminus V(\textbf{F})|}p_\eta^{|E(\mathbf{T})\setminus E(\mathbf{F})|}=\sum_{\mathbf L\subset \mathbf F\subsetneq \mathbf T}n^{|V(\mathbf T)\setminus V(\mathbf F)|-\alpha_\eta |E(\mathbf T)\setminus E(\mathbf F)|}\,, \label{eq-unique-is-typical}
    \end{align}
    which is $o(1)$ by \eqref{eq:balanced}, as desired.
\end{proof}

The rest of this subsection is devoted to the proof of Lemma~\ref{lem:structure}. Recalling the definition of $\mathsf{GO}_{s-1}$, we see $\widehat{\Pb}$ is a product measure on the space $\Omega=\{0,1\}^{\mathsf E_0 \setminus E(\mathsf{GO}_{s-1})}$. Denote $\mathsf{T}\cong \textbf T$ as the subtree of $\Gs$ that certifies $\{\Pi(\operatorname L)\Join_\Gs I_\Gs(\operatorname Q)\}$ as before. For $\omega\in \Omega$, we write $\omega=\omega_1\oplus \omega_2$, where $\omega_1\in \Omega_1 = \{0, 1\}^{\mathsf E_0 \setminus (E(\mathsf{GO}_{s-1}) \cup E(\mathsf T))}$ and $\omega_2 \in \Omega_2 = \{0, 1\}^{E(\mathsf T)}$ (here $\oplus$ means concatenation). It is then clear from the definition that 
	\[
	\widehat{\Pb}[\omega_1\oplus \omega_2]=\widehat{\Pb}[\omega_1]\widehat\Pb[\omega_2] \mbox{ and } \{\Pi(\operatorname L)\Join_{\Gs} I_\Gs(\operatorname Q)\} = \{\omega_2=(1,\ldots,1)\}\,.
	\]
For $i\in \{0, 1\}$, define
\begin{equation}
	\label{def:mathscrA}
	\mathscr{A}_s^{i}=\left\{\omega_1\in \Omega_1:\omega_1\oplus\{i,\dots,i\}\in \mathcal A_s^1\right\}\,.
\end{equation}
 The fact that $\mathcal A_s^1$ is decreasing implies that $\mathscr{A}_s^{0}$ is also decreasing and $\mathscr{A}_{s}^{1}\subset \mathscr{A}_s^{0}$. The left hand side of \eqref{eq: 5 goal first moment} can be expressed as $\widehat{\Pb}\big[\mathscr{A}_s^{1}\big]$.
 For $\widehat{\Pb}\big[\mathcal A_s^1\big]$ we have
\begin{align*}
		\widehat{\Pb}\left[\mathcal A_s^1\right]=&\ \sum_{\omega_1\in\Omega_1}\sum_{\omega_2\in\Omega_2}\mathbf{1}_{\mathcal A_s^1}( \omega_1\oplus\omega_2)\widehat{\Pb}[\omega_1\oplus\omega_2]\\
		\le &\ \sum_{\omega_1\in\Omega_1}1_{\mathcal A_s^1}(\omega_1\oplus\{0,\dots,0\}){\Pb}[\omega_1]\sum_{\omega_2\in \Omega_2}\Pb[\omega_2]\\
		=&\ \widehat{\Pb}\left[\big\{\omega_1:\omega_1\oplus\{0,\dots,0\}\in \mathcal A_s^1\big\}\right]=\widehat{\Pb}\left[\mathscr{A}_s^{0}\right]\,.
\end{align*}
Provided with this, we see \eqref{eq: 5 goal first moment} reduces to 
\begin{equation*}
	\frac{\widehat\Pb[\mathscr{A}_s^{1}]}{\widehat{\Pb}[\mathscr{A}_s^{0}]}=\widehat{\Pb}\big[\mathscr{A}_s^{1}\mid \mathscr{A}_s^{0}\big]\ge 1-o(1)\,.
\end{equation*}
To this end, we note that for $\omega_1 \in \mathscr A^0_s \setminus \mathscr A^1_s$, we have $\mathcal A_s^1$ holds if the edges in $E(\mathsf{T})$ are all closed while $\mathcal A_s^1$ fails after opening these edges. There are two possible scenarios for this: (i) opening edges in $\mathsf{T}$ changes the realization of $\mathcal{F}_{s-1}$ and hence alters the event $\mathcal A_s^1$; (ii) opening edges in $\mathsf{T}$ does not change the realization of $\mathcal{F}_{s-1}$, but some of these edges participate in the tuple which certifies the failure of $\mathcal A_s^1$. Denote by $\mathcal E_s$ the event that after opening edges in $E(\mathsf{T})$, there exists a subgraph $\mathsf{T}' \subset \Gs$ with an isomorphism $\phi:\mathbf T\to \mathsf T'$ such that
\begin{equation}\label{eq-def-mathcal-E-s}
E(\mathsf T') \cap E(\mathsf T) \neq \emptyset \mbox{ and } \phi(\mathbf L) = \Pi(\operatorname L') \mbox{ for some }\operatorname L'\in\bigcup_{1\le t \le s-1}\operatorname{EXP}_{t}\,.
\end{equation}
Clearly, $\mathcal E_s$ is an increasing event. We claim that both scenarios imply $\mathcal E_s$. Indeed, if $\mathcal E_s$ does not hold, then after opening edges in $E(\mathsf{T})$, the realization of $\mathcal{F}_{s-1}$ remains to be the same by the rule of Algorithm~\ref{algo:greedy}. So, in particular $\mathsf{Fail}_{s-1}$ remains to be the same. In addition, from the definition of $\mathcal A_s^1$, we see under the event $\mathcal E_s^c$, whether $\mathcal A_s^1$ happens does not depend on openess/closedness for edges in $E(\mathsf{T})$. This proves the claim and now it suffice to show 
\[
\widehat{\Pb}\big[\mathcal E_s\mid \mathscr{A}_s^{0}\big]\le \widehat{\Pb}[\mathcal E_s]=o(1)\,,
\]
where the first inequality follows from FKG. 

We divide $\mathcal E_s$ according to the intersecting patterns of $\mathsf{T}'$ with $\mathsf T$ and $ \mathsf{GO}_{s-1}$. Denote $\mathcal P$ for the pairs $(\mathbf{F}_1,\mathbf{F}_2)$ with $\mathbf{F}_1,\mathbf{F}_2\subset \mathbf T$ such that $\mathbf{F}_1$ is a \emph{proper} subtree of $\mathbf T$ with at least one edge, $V(\mathbf F_1)\cap V(\mathbf F_2)=\emptyset$ and $ \mathbf L \subset V(\mathbf{F}_1)\cup V( \mathbf F_2)$. For a pair $(\mathbf F_1,\mathbf F_2)\in \mathcal P$, define $\mathcal E_s(\mathbf F_1,\mathbf F_2)$ to be the event that \eqref{eq-def-mathcal-E-s} holds and that
\[
\phi\big(E(\mathbf F_1)\big)\cap E(\mathsf{T})\neq \emptyset\,,\quad \phi\left(\mathbf F_1\cup\mathbf F_2\right)= \mathsf{T}'\cap (\mathsf{GO}_{s-1}\cup \mathsf T)\,.
\]

Recalling \eqref{eq-def-mathcal-E-s}, we obtain 
\begin{equation}
	\label{eq:containing-relation-naive-version}
	\mathcal E_s\subset \mathcal{E}_s(\mathbf T,\emptyset)\cup \bigcup_{(\mathbf{F}_1,\mathbf{F}_2)\in \mathcal P}\mathcal E_s(\textbf{F}_1,\mathbf{F}_2)\,.
\end{equation}
(Indeed, $\mathbf F_1$ can be taken as any component of $\phi^{-1}(\mathsf T'\cap (\mathsf T\cup \mathsf{GO}_{s-1}))$ which contains at least one edge in $E(\mathsf T)$.) 
We want to exclude the case $\mathcal E_s(\mathbf T,\emptyset)$ from the assumption that $\operatorname{L}$ is $s$-good. To this end, we make the following definition.
	\begin{defn}
		\label{def:s-good}
		For a $\xi$-tuple $\operatorname{L}\in \mathfrak{A}(\operatorname{M}_{s-1},\xi)$, we say it is \emph{$s$-bad}, if for some $T \cong \mathbf T$ with leaf set $L=I_G(\operatorname{L})$ and $V(T)\cap M_{s-1}=L$, we have that the graph $T\cup GO_{s-1}$ contains a subgraph $T^*\cong \mathbf T$ satisfying that $E(T^*)\cap E(T)\neq \emptyset$ and that the leaf set of $T^*$ equals to $L^*=I_G(\operatorname{L}^*)$ for some $\operatorname{L}^*\in \bigcup_{1\le t\le s-1}\operatorname{EXP}_{t}$ (note that in this definition for `some' $T\cong \mathbf T$ is equivalent to for `any' $T\cong \mathbf T$). Otherwise we say $\operatorname{L}\in \mathfrak{A}(\operatorname{M}_{s-1},\xi)$ is \emph{$s$-good}.
	\end{defn}
From the definition and the fact that $GO_{s-1}\cong \mathsf{GO}_{s-1}$ through $\Pi\circ I_G^{-1}$, we see that $\mathcal E_s\cap \mathcal E_s(\mathbf T,\emptyset)=\emptyset$  under the assumption that $\operatorname{L}$ is $s$-good. Thus, \eqref{eq:containing-relation-naive-version} can be strengthened as
\begin{equation}
	\label{eq:containing-relation}
	\mathcal E_s\subset \bigcup_{(\mathbf{F}_1,\mathbf{F}_2)\in \mathcal P}\mathcal E_s(\mathbf{F}_1,\mathbf{F}_2)\,.
\end{equation}

For a pair $(\mathbf F_1,\mathbf F_2)\in \mathcal P$, let $\operatorname{Enum}(\mathbf F_1,\mathbf F_2)$ be the number of $\xi$-tuples $\operatorname L'\in \bigcup_{1\le t\le s-1}\operatorname{EXP}_{t}$ such that there exist two 
embeddings $\phi_i:\mathbf F_i\to\mathsf{T}\cup \mathsf{GO}_{s-1},i=1,2$ 
such that
 $$\phi_1\big(E(\mathbf F_1)\big)\cap E(\mathsf T)\neq \emptyset\text{ and }\phi_1\big(V(\mathbf F_1)\cap \mathbf L\big)\cup\phi_2\big(V(\mathbf F_2)\cap \mathbf L\big)= \Pi(\operatorname L')\,.$$ 
 (Note that $ \operatorname {Enum}(\mathbf F_1,\mathbf F_2)$ depends on $\mathsf T$ although we did not include $\mathsf T$ in the notation.)
Then $ \operatorname {Enum}(\mathbf F_1,\mathbf F_2)$ is the number of possible choices for the leaf set of $\mathsf{T}'$ which may potentially certify the event $\mathcal E_s(\mathbf F_1,\mathbf F_2)$. For each fixed choice, we see from \eqref{eq:observation} that the probability that this indeed certifies the event $\mathcal E_s(\mathbf F_1,\mathbf F_2)$ is upper-bounded by 
\begin{equation}
	\label{eq:Prob}
	 \operatorname {Prob}(\mathbf F_1,\mathbf F_2)=n^{|V(\mathbf T)|-|V(\mathbf F_1)|-|V(\mathbf F_2)|}p_\eta^{|E(\mathbf T)|-|E(\mathbf F_1)|-|E(\mathbf F_2)|}\,.
\end{equation}
Then by a union bound, for any pair $(\mathbf F_1,\mathbf F_2)\in \mathcal P$ it holds that 
\begin{equation}
	\label{eq:prob-times-enum}
	\widehat{\Pb}\left[\mathcal E_s(\mathbf F_1,\mathbf F_2)\right]\le  \operatorname {Enum}(\mathbf F_1,\mathbf F_2)\times  \operatorname {Prob}(\mathbf F_1,\mathbf F_2)\,.
\end{equation}
Let $\mathcal P_0$ be the collection of $(\mathbf F_1,\mathbf F_2) \in \mathcal P$ which maximizes the right hand side of \eqref{eq:prob-times-enum}. We claim that for any $(\mathbf F_1,\mathbf F_2) \in \mathcal P_0$ the components of $\mathbf F_2$ are dense trees which intersect $\mathbf L$.
We now prove this claim by contradiction, and we divide the proof into two cases. 
If $\mathbf F_2$ contains a tree component disjoint from $\mathbf L$, we can simply remove this component from $\mathbf F_2$ to get a pair $(\mathbf F_1,\mathbf F_2')\in \mathcal P$ with 
\[
\operatorname {Prob}(\mathbf F_1,\mathbf F_2')> \operatorname {Prob}(\mathbf F_1,\mathbf F_2),\quad \operatorname {Enum}(\mathbf F_1,\mathbf F_2')\ge  \operatorname {Enum}(\mathbf F_1,\mathbf F_2)\,,
\]
contradicting the maximality. If $\mathbf F_2$ contains a tree component $\mathbf T_0$ which intersects $\mathbf L$ but is not dense, we may find some $\mathbf F_0\subsetneq \mathbf T_0$  which contains $V(\mathbf T_0)\cap \mathbf L$ such that $\operatorname {Cap}(\mathbf F_0)< \operatorname {Cap}(\mathbf T_0)$ (this is feasible since $\alpha_\eta$ is irrational). We define $\mathbf F'_2$ by replacing $\mathbf T_0$ with $\mathbf F_0$ in $\mathbf F_2$ and get a pair $(\mathbf F_1,\mathbf F_2')\in \mathcal P$. Again, it satisfies
\[
 \operatorname {Prob}(\mathbf F_1,\mathbf F_2')>\operatorname {Prob}(\mathbf F_1,\mathbf F_2) \mbox{ and } \operatorname {Enum}(\mathbf F_1,\mathbf F_2')\ge  \operatorname {Enum}(\mathbf F_1,\mathbf F_2)\,,
\] 
contradicting the maximality. This completes the verification of the claim. 
Recall \eqref{eq:def-of-M(T)}. We  are now ready to define our good event $\Gc_s^1$.
\begin{defn}
	\label{def:G-s-1}
	Fix some large constant $\kappa_1=\kappa_1(\eta,\alpha_\eta,\mathbf T)$ which will be determined later. We define $\Gc_s^1$ to be the event that for any  $(\mathbf F_1,\mathbf F_2)\in \mathcal P_0$ and for any $\mathsf T\subset  \mathsf G$,
	\begin{equation}
		\label{eq:good-s-1}
		 \operatorname {Enum}(\mathbf F_1,\mathbf F_2)\le \begin{cases}
			\kappa_1n\times(np_\eta)^{|E(\mathbf T)|-|E(\mathbf F_2)|},\quad &\text{if }V(\mathbf F_1)\cap \mathbf L = \emptyset\,,\\
			\kappa_1\mathbb D(\mathbf F_1)\times (np_\eta)^{|E(\mathbf T)|-|E(\mathbf F_1)|-|E(\mathbf F_2)|},&\text{if }V(\mathbf F_1)\cap \mathbf L \neq \emptyset\,.
		\end{cases}
	\end{equation} 
\end{defn}
Since $|\mathcal P|$ is uniformly bounded in $n$, it suffices to show that on $\Gc_s^1$
\begin{equation}\label{eq-enum-prob-P-0}
\operatorname {Enum}(\mathbf F_1,\mathbf F_2)\times  \operatorname {Prob}(\mathbf F_1,\mathbf F_2) = o(1) \mbox{ for any }(\mathbf F_1, \mathbf F_2)\in \mathcal P_0\,.
\end{equation}
To this end, we may write
$
(\mathbf F_1,\mathbf F_2)=(\mathbf T_0,\mathbf T_1\cup \cdots\cup \mathbf T_l)
$,
where $\mathbf T_i$ is a subtree of $\mathbf T$ for $0\le i\le l$. The proof is divided into two cases.

In the case that $V(\mathbf T_0)\cap \mathbf L = \emptyset$, the target product is upper-bounded by (recalling \eqref{eq:good-s-1})
\begin{align*}
&\kappa_1n^{\chi+\zeta}p_\eta^{2\zeta}\times n^{-|V(\mathbf T_0)|+1}p_\eta^{-|E(\mathbf T_0)|}\times \prod_{i=1}^{l}n^{-|V(\mathbf T_i)\cap \mathbf Q|-|E(\mathbf T_i)|}p_\eta^{-2|E(\mathbf T_i)|}\\
=&\ \kappa_1n^{\chi-(2\alpha_\eta-1)\zeta}\times (np_\eta)^{-|E(\mathbf T_0)|}\times \prod_{i=1}^l n^{-|V(\mathbf T_i)\cap \mathbf Q|+(2\alpha_\eta-1)|E(\mathbf T_i)|}\,.
\end{align*}
Note that the second term (i.e., $(np_\eta)^{-|E(\mathbf T_0)|}$) is no more than $n^{\alpha_\eta-1}$ since $E(\mathbf T_0)\neq \emptyset$ and 
that the third term is bounded by $1$ from \eqref{eq:balanced-2}. Thus, we see the expression above is upper-bounded by $\kappa_1n^{\chi-(2\alpha_\eta-1)\zeta}\times n^{\alpha_\eta-1}$, which is $o(1)$ by Lemma~\ref{lem:designing-tree} (i). This proves the desired bound in this case. 

In the case that $V(\mathbf T_0)\cap \mathbf L \neq \emptyset$, we may write 
\[
\mathbb D(\mathbf T_0)=n^{\operatorname {Cap}(\mathbf T_0)- \operatorname {Cap}(\mathbf F_0)}=n^{\operatorname {Cap}(\mathbf T_0)-\sum_{j=1}^m \operatorname {Cap}(\mathbf T_j')}\,,
\]
where $\mathbf F_0 \subset \mathbf T_0$ is the union of disjoint trees $\mathbf T_1', \ldots, \mathbf T_m'$ with  $V(\mathbf T_0)\cap \mathbf L \subset V(\mathbf F_0)$ such that $ \operatorname {Cap}(\mathbf F_0)$ is minimized. As a result, we see the product $ \operatorname {Enum}(\mathbf F_1,\mathbf F_2)\times \operatorname {Prob}(\mathbf F_1,\mathbf F_2)$ is upper-bounded by
\begin{equation}\label{eq-bound-long-expression}
	\kappa_1n^{\chi+\zeta}p_\eta^{2\zeta}\times n^{-\sum_{j=1}^m \operatorname {Cap}(\mathbf T_j')}\times (np_\eta)^{-|E(\mathbf T_0)|}\times \prod_{i=1}^{l}n^{-|V(\mathbf T_i)\cap \mathbf Q|-|E(\mathbf T_i)|}p_\eta^{-2|E(\mathbf T_i)|}\,.
	\end{equation}
	If some of $\mathbf T_j'$  is not a singleton, then \eqref{eq-bound-long-expression} is upper-bounded by (recalling that the product above is upper-bounded by 1)
	\begin{align*}
		&\ \kappa_1n^{\chi+\zeta}p_\eta^{2\zeta}\times n^{-\sum_{j=1}^m  \operatorname {Cap}(\mathbf T_j')}\times (np_\eta)^{-\sum_{j=1}^m|E(\mathbf T_j')|}\\
		\leq&\  \kappa_1n^{\chi-(2\alpha_\eta-1)\zeta}\prod_{j=1}^m n^{-|V(\mathbf T_j')\cap Q(\mathbf T)|+(2\alpha_\eta-1)|E(\mathbf T_j')|}\,,
		\end{align*}
		which is $o(1)$ by Lemma~\ref{lem:designing-tree} (i) and (iv). If each $\mathbf T_j'$ is a singleton, then \eqref{eq-bound-long-expression} is upper-bounded by
		$$\kappa_1n^{\chi+\zeta}p_\eta^{2\zeta}\times (np_\eta)^{-1} \leq \kappa_1n^{\chi-(2\alpha_\eta-1)\zeta}\times n^{\alpha_\eta-1}\,,$$
which is also $o(1)$ Lemma~\ref{lem:designing-tree} (i). This completes the proof in this case, and thus finally completes the proof of Lemma~\ref{lem:structure}.

\subsection{Proof of Proposition~\ref{prop:two-point-est}}

We continue to fix some $1\le s\le S$ and a realization of $\mathcal{F}_{s-1}$. We further fix a realization of $\operatorname {CAND}_s$, and hence we also get the set of $s$-good $\xi$-tuples in $\operatorname{CAND}_s$, denoted as $\operatorname{GC}_s=\{\operatorname L_1,\dots,\operatorname L_l\}$.
For any nonempty subset $\mathbf R \subset \mathbf L$, denote $\mathbf{Span}(\mathbf R)$ for the subtree of $\mathbf T$ spanned by $\mathbf R$ (i.e., the minimal subtree that contains $\mathbf R$). Throughout this section, it will be convenient to partition pairs in $\operatorname {GC}_s\times \operatorname {GC}_s$ according to their intersecting patterns. More precisely, for two $\xi$-tuples $\operatorname L = (t_1, \ldots, t_\xi)$ and $\operatorname L' = (t'_1, \ldots, t'_\xi)$, we let $\operatorname {Loc}(\operatorname L,\operatorname L') = \{1\leq i\leq \xi: t'_i \in \{t_1, \ldots, t_\xi\}\}$. For any subset $\mathbf R\subset \mathbf L$ (Recall that $\mathbf L=\{\chi+1,\dots,\chi+\xi\}$), we let 
\[
 \operatorname {IP}_s(\mathbf R)=\big\{(\operatorname L_i,\operatorname L_j) \in\operatorname{GC}_s \times \operatorname{GC}_s :\operatorname {Loc}(\operatorname L_i,\operatorname L_j)=\{r-\chi, r\in \mathbf R\}\big\}\,.
\]
We are now ready to define our good event $\Gc_s^2$.
\begin{defn}
	\label{def:G-s-2}
Fix some positive constants $\kappa_2,\kappa_3$ depending only on $\eta,\alpha_\eta$ and $\mathbf T$ which will be determined later. The good event $\Gc_s^2$ is the intersection of the following two events:\\
(i) $l=|\operatorname{GC}_s|\ge \kappa_2(np_\eta)^\zeta$.\\
(ii) For any nonempty subset $\mathbf R\subset \mathbf L$, 
$$|\operatorname {IP}_s(\mathbf R)| \leq \kappa_3(np_\eta)^\zeta\times \mathbb D(\mathbf{Span}(\mathbf R))\times (np_\eta)^{|E(\mathbf T)\setminus E(\mathbf{Span}(\mathbf R))|}\,.$$
\end{defn}
We next assume that the realization $\mathcal{F}_{s-1/2}=\sigma(\mathcal{F}_s\cup \operatorname{CAND}_s)$ satisfies $\Gc_s^1\cap \Gc_s^2$ and prove \eqref{eq:second-moment-bound}. For simplicity, we write $\mathsf L_i=\Pi(\operatorname L_i)$. Since $\mathcal{F}_{s-1}$ satisfies $\Gc_s^1$, from Proposition~\ref{prop:one-point-est} we see for any $\mathsf L_i\in \Pi\left(\operatorname {GC}_s\right)$,
\[
\Pb[\mathsf L_i\Join_{\Gs} \mathsf R_{s-1}\mid \mathcal{F}_{s-1/2}]=\Pb[\mathsf L_i\Join_{\mathsf G}\mathsf R_{s-1}\mid \mathcal{F}_{s-1}]\ge\big[1-o(1)\big]\Pb[\mathsf L_i\Join_{\Gs} \mathsf R_{s-1}]\ge c_1n^\chi p_\eta^\zeta\,.
\]
Combined with (i) in $\Gc_s^2$, this yields a lower bound on the right hand side of \eqref{eq:second-moment-bound}:
\begin{equation}
	\label{eq:lower-bound-of-EXs}
\begin{split}
		\big(\mathbb{E}\left[X_s\mid\mathcal{F}_{s-1/2}\right]\big)^2=\left(\sum_{i=1}^l \Pb[\mathsf L_i\Join_{\Gs} \mathsf R_{s-1}\mid \mathcal{F}_{s-1/2}]\right)^2\gtrsim l^2\big(n^{\chi}p_\eta^{\zeta}\big)^2\gtrsim \big(n^{\chi+\zeta}p_\eta^{2\zeta}\big)^2\,.
\end{split}
\end{equation}
For the left hand side of \eqref{eq:second-moment-bound}, we may expand it out and break the sum into several parts as follows:
\begin{align}
		\nonumber\ \mathbb{E}\left[X_s^2\mid \mathcal{F}_{s-1/2}\right]=&\sum_{i,j=1}^l \Pb[\mathsf L_i\Join_{\Gs} \mathsf R_{s-1},\mathsf L_j\Join_{\Gs}\mathsf R_{s-1}\mid \mathcal{F}_{s-1/2}]\\
		=&\sum_{(\operatorname L_i,\operatorname L_j)\in\operatorname{IP}_s(\emptyset)}\Pb[\mathsf L_i\Join_\Gs\mathsf R_{s-1} ,\mathsf L_j\Join_\Gs\mathsf R_{s-1}\mid \mathcal{F}_{s-1}] \label{eq:break-of-sum-1}\\
		+&\sum_{\emptyset\neq \mathbf R\subset \mathbf L}\sum_{(\operatorname L_i,\operatorname L_j)\in\operatorname {IP}_s(\mathbf R)}\Pb[\mathsf L_i\Join_\Gs\mathsf R_{s-1},\mathsf L_j\Join_\Gs\mathsf R_{s-1}\mid \mathcal{F}_{s-1}]\,.\label{eq:break-of-sum-2}
\end{align}
We estimate the probabilities in the sum above by the following two lemmas.
\begin{lemma}
	\label{lem:est-for-disjoint}
	For any  $(\operatorname L_i,\operatorname L_j)\in\operatorname{IP}_s(\emptyset)$, it holds that
	\begin{equation}
		\label{eq:est-for-disjoint}
		\begin{aligned}
		&\ \Pb[\mathsf L_i\Join_{\Gs} \mathsf R_{s-1}, \mathsf L_j\Join_\Gs\mathsf R_{s-1}\mid \mathcal{F}_{s-1}]\\
		\le&\ \big[1+o(1)\big]\Pb[\mathsf L_i\Join_\Gs\mathsf R_{s-1}\mid \mathcal{F}_{s-1}]\Pb[\mathsf L_j\Join_\Gs\mathsf R_{s-1}\mid \mathcal{F}_{s-1}]\,.
		\end{aligned}
	\end{equation}
\end{lemma}
\begin{lemma}
	\label{lem:est-for-intersecting}
	For any nonempty subset $\mathbf R \subset \mathbf L$, let $\mathbf{F}_{\mathbf R}$ be the subgraph of $\mathbf{T}$ with $V(\mathbf F_{\mathbf R})\cap \mathbf L = \mathbf R$ such that $ \operatorname{Cap}(\mathbf F_{\mathbf R})$ is minimized out of all such subgraphs. Then for any  $(\operatorname L_i,\operatorname L_j)\in\operatorname {IP}_s(\mathbf R)$, it holds that
	\begin{equation}
		\label{eq:est-for-intersecting}
		\Pb[\mathsf L_i\Join_\Gs\mathsf R_{s-1},\mathsf L_j\Join_\Gs\mathsf R_{s-1}\mid \mathcal{F}_{s-1}]\lesssim \big(n^\chi p_\eta^\zeta\big)^2\times n^{- \operatorname {Cap}(\mathbf{F}_{\mathbf R})}. 
	\end{equation}
\end{lemma}
\begin{proof}[Proof of Proposition~\ref{prop:two-point-est}]
From Lemma~\ref{lem:est-for-disjoint}, \eqref{eq:break-of-sum-1} is upper-bounded by
\begin{align}
	\big[1+o(1)&\big]\sum_{(\operatorname L_i,\operatorname L_j)\in\operatorname {IP}_s(\emptyset)}\Pb[\mathsf L_i\Join_\Gs\mathsf R_{s-1}\mid  \mathcal F_{s-1}]\Pb[\mathsf L_j\Join_\Gs\mathsf R_{s-1}\mid \mathcal F_{s-1}]\\
	\le \big[1+o(1)&\big]\sum_{i,j=1}^l\Pb[\mathsf L_i\Join_\Gs\mathsf R_{s-1} \mid \mathcal{F}_{s-1}]\Pb[\mathsf L_j\Join_\Gs\mathsf R_{s-1} \mid \mathcal{F}_{s-1}]
\end{align}
which is $\big[1+o(1)\big]\big(\mathbb{E}\left(X_s\mid \mathcal{F}_{s-1}\right)\big)^2$. In addition, for each nonempty subset $\mathbf R\subset \mathbf L$, by (ii) in $\Gc_s^2$ and Lemma~\ref{lem:est-for-intersecting} we see that $\sum_{(\operatorname L_i,\operatorname L_j)\in\operatorname{IP}_s(\mathbf R)}\Pb[\mathsf L_i\Join_\Gs\mathsf R_{s-1},\mathsf L_j\Join_\Gs\mathsf R_{s-1}\mid\mathcal{F}_{s-1}]$ is bounded by a constant multiplies
\begin{align}
\nonumber &\ (np_\eta)^\zeta\times\mathbb D(\textbf{Span}(\mathbf R))\times (np_\eta)^{|E(\mathbf T)\setminus E(\mathbf{Span}(\mathbf R))|}\times \big(n^\chi p_\eta^\zeta\big)^2\times n^{- \operatorname {Cap}(\mathbf F_{\mathbf R})}\\
=&\ \big(n^{\chi+\zeta}p_\eta^{2\zeta}\big)^2\times n^{-|E(\textbf{Span}(\mathbf R))|+|V(\textbf{Span}(\mathbf R))\cap \mathbf Q|}\times n^{-2 \operatorname {Cap}(\mathbf F_{\mathbf R})}\,,	\label{eq:est-of-intersecting-sum}
\end{align}
where we used the fact that $\mathbb D(\mathbf{Span}(\mathbf R))=n^{\operatorname {Cap}(\mathbf{Span}(\mathbf R))-\operatorname  {Cap}(\mathbf F_{\mathbf R})}$. Suppose $\mathbf F_{\mathbf R}$ is a union of subtrees $\mathbf T_1,\dots,\mathbf T_r$ of $\mathbf T$. We note that
\[
|E(\mathbf{Span}(\mathbf R))|-|V(\mathbf{Span}(\mathbf R))\cap \mathbf Q| = 
|\mathbf R| - 1 \geq |\mathbf R| - r = \sum_{i=1}^r\big(|E(\mathbf T_i)|-|V(\mathbf T_i)\cap \mathbf Q|\big)\,,
\]
Thus, the term $n^{-|E(\mathbf{Span}(\mathbf R))|+|V(\mathbf{Span}(\mathbf R))\cap \mathbf Q|}\times n^{-2 \operatorname {Cap}(\mathbf F_{\mathbf R})}$ in \eqref{eq:est-of-intersecting-sum} is bounded by
\begin{align*}
	 \prod_{i=1}^r n^{-|E(\mathbf T_i)|+|V(\mathbf T_i)\cap \mathbf Q|-2 \operatorname {Cap}(\mathbf T_i)}
	=\prod_{i=1}^r n^{-|V(\mathbf T_i)\cap \mathbf Q|+(2\alpha_\eta-1)|E(\mathbf T_i)|}\,,
\end{align*}
which is $o(1)$ from  Lemma~\ref{lem:designing-tree} (iv). This shows that \eqref{eq:est-of-intersecting-sum} is $o\Big(\big(n^{\chi+\zeta}p_\eta^{2\zeta}\big)^2\Big)$ for each $\mathbf R$.  
Since the number of possible $\mathbf R$ is uniformly bounded in $n$, we complete the proof by combining with \eqref{eq:lower-bound-of-EXs}. 
\end{proof}
It remains to prove Lemma~\ref{lem:est-for-disjoint} and Lemma~\ref{lem:est-for-intersecting}. We note that for any two tuples $\mathsf L_i,\mathsf L_j\in \Pi\left(\operatorname{GC}_s\right)$, it holds
\begin{align}
&\ \Pb[\mathsf L_i\Join_{\Gs} \mathsf R_{s-1},\mathsf L_j\Join_\Gs\mathsf R_{s-1}\mid \mathcal{F}_{s-1}]=\widehat{\Pb}[\mathsf L_i\Join_\Gs\mathsf R_{s-1},\mathsf L_j\Join_\Gs\mathsf R_{s-1}\mid \mathcal A_s^1] \nonumber\\
\leq&\ \widehat{\Pb}[\mathsf L_i\Join_\Gs\mathsf R_{s-1},\mathsf L_j\Join_\Gs\mathsf R_{s-1}]= \Pb[\mathsf L_i\Join_\Gs\mathsf R_{s-1},\mathsf L_j\Join_\Gs\mathsf R_{s-1}]\,,
\end{align}
where the inequality follows from the FKG inequality and the last equality follows from independence. Then it is clear that $\Pb[\mathsf L_i\Join_\Gs\mathsf R_{s-1},\mathsf L_j\Join_\Gs\mathsf R_{s-1} \mid \mathcal F_{s-1}]$ is upper-bounded by
	\begin{align}
	&\sum_{\operatorname Q\in \mathfrak{A}(\operatorname R_{s-1}',\chi)}\Pb[\mathsf L_i\Join_{\Gs} I_\Gs(\operatorname Q),\mathsf L_j\Join_\Gs\mathsf R_{s-1}] \nonumber\\
	=&\sum_{\operatorname Q\in \mathfrak{A}(\operatorname R_{s-1}',\chi)}\Pb[\mathsf L_i\Join_\Gs I_\Gs(\operatorname Q)]\Pb[\mathsf L_j\Join_\Gs\mathsf R_{s-1}\mid \mathsf L_i\Join_\Gs I_\Gs(\operatorname Q)]\,.
	\end{align}
By Proposition~\ref{prop:one-point-est} and \eqref{eq-to-cite-prop-2.3-proof}, we have
	\begin{equation}
		\label{eq:prob=first-moment}
	\sum_{\operatorname Q\in \mathfrak{A}(\operatorname R_{s-1}',\chi)}\Pb[\mathsf L_i\Join_\Gs I_\Gs(\operatorname Q)]\le \big[1+o(1)\big]\Pb[\mathsf L_i\Join_\Gs\mathsf R_{s-1}\mid\mathcal{F}_{s-1}]\lesssim n^\chi p_\eta^\zeta\,.
	\end{equation}
	Thus it remains to show for each $\operatorname Q\in \mathfrak{A}(\operatorname R_{s-1}',\chi)$, we have
	\[
	\Pb[\mathsf L_j\Join_\Gs\mathsf R_{s-1} \mid \mathsf L_i\Join_\Gs I_\Gs(\operatorname Q)]\le\begin{cases}
		\big[1+o(1)\big]\Pb[\mathsf L_j\Join_\Gs\mathsf R_{s-1}]\,,&\text{ if }(\operatorname L_i,\operatorname L_j)\in\operatorname{IP}_s(\emptyset)\,;\\
		O\big(n^\chi p_\eta^\zeta\times n^{- \operatorname {Cap}(\mathbf F_{\mathbf R})}\big)\,,&\text{ if }(\operatorname L_i,\operatorname L_j)\in\operatorname{IP}_s(\mathbf R), \mathbf R\neq \emptyset\,.
	\end{cases}
	\]
	Fix such a tuple $\operatorname Q$ and denote $\mathsf T \subset \Gs$ the subtree that certifies $\{\mathsf L_i\Join_\Gs I_\Gs(\operatorname Q)\}$. Similarly for each $\operatorname Q'\in \mathfrak{A}(\operatorname R_{s-1}',\chi)$, we let $\mathsf T' \subset \Gs$ be the subtree that certifies $\{\mathsf L_j\Join_{\Gs} I_\Gs(\operatorname Q')\}$. Then from a union bound we get
	\begin{align}
	&\Pb[\mathsf L_j\Join_\Gs\mathsf R_{s-1} \mid \mathsf L_i\Join_\Gs I_\Gs(\operatorname Q)]\le\ \sum_{\operatorname Q'\in \mathfrak{A}(\operatorname R_{s-1}',\chi)}\Pb[\mathsf L_j\Join_\Gs I_\Gs(\operatorname Q')\mid \mathsf L_i\Join_\Gs I_\Gs(\operatorname Q)] \nonumber\\
	=&\ \sum_{\substack{\operatorname Q'\in \mathfrak{A}(\operatorname R_{s-1}',\chi)\\E(\mathsf T)\cap E(\mathsf T')=\emptyset}}\Pb[\mathsf L_j\Join_\Gs I_\Gs(\operatorname Q')]+\sum_{\substack{\operatorname Q'\in \mathfrak{A}(\operatorname R_{s-1}',\chi)\\E(\mathsf{T})\cap E(\mathsf T')\neq \emptyset}}\Pb[\mathsf L_j\Join_\Gs I_\Gs(\operatorname Q')\mid \mathsf L_i\Join_\Gs I_\Gs(\operatorname Q)]\,,\label{eq-to-cite-sums}
	\end{align}
where the equality follows from independence. 

For any subset $\mathbf R \subset \mathbf L$ (possibly $\mathbf R = \emptyset$), we define $\mathcal P(\mathbf R)$ as the collection of nonempty subgraphs $\mathbf F \subset \mathbf T$ with $V(\mathbf F)\cap \mathbf L=\mathbf R$. 
\begin{proof}[Proof of Lemma~\ref{lem:est-for-disjoint}]
For the case $(\operatorname L_i,\operatorname L_j)\in\operatorname {IP}_s(\emptyset)$, we have $\mathsf L_i\cap \mathsf L_j=\emptyset$. The first sum in \eqref{eq-to-cite-sums} is bounded by $\big[1+o(1)\big]\Pb[\mathsf L_j\Join_\Gs\mathsf R_{s-1}\mid \mathcal{F}_{s-1}]$ for the same reason as \eqref{eq:prob=first-moment}. As for the second sum in \eqref{eq-to-cite-sums}, it can be upper-bounded by
\[
\sum_{\mathbf F\in \mathcal P(\emptyset)}\sum_{\substack{{\operatorname Q'\in \mathfrak{A}(\operatorname R_{s-1}',\chi)}\\ \mathsf T \cap \mathsf T'\cong \mathbf F}}\Pb[\mathsf L_j\Join_{\Gs}I_\Gs(\operatorname Q')\mid \mathsf L_i\Join_\Gs I_\Gs(\operatorname Q)]\,.
\]
For each $\mathbf F\in \mathcal P(\emptyset)$, it is readily to see the second summation above 
is bounded by
\[
n^{|\mathbf Q\setminus V(\mathbf F)|}p_\eta^{|E(\mathbf T)\setminus E(\mathbf F)|}=n^\chi p_\eta^\zeta\times n^{-|V(\mathbf F)|}p_\eta^{-|E(\mathbf F)|}=o\big(n^\chi p_\eta^\zeta\big)=o\big(\Pb[\mathsf L_j\Join_\Gs\mathsf R_{s-1}\mid\mathcal{F}_{s-1}]\big)\,,
\] 
where we used the fact that $|V(\mathbf F)|>|E(\mathbf F)|$ for any $\mathbf F\in \mathcal P(\emptyset)$ since $\mathbf F$ is a forest. Since the cardinality of $\mathcal P(\emptyset)$ is uniformly bounded in $n$, this concludes Lemma~\ref{lem:est-for-disjoint}.
\end{proof}
\begin{proof}[Proof of Lemma~\ref{lem:est-for-intersecting}]
 For the case $(\operatorname L_i,\operatorname L_j)\in\operatorname{IP}_s(\mathbf R)$ with $\mathbf R\neq \emptyset$, similarly, the first sum in \eqref{eq-to-cite-sums} remains to be $O(n^\chi p_\eta^\zeta)$. Note that the second sum can be expressed as
\[
\sum_{\mathbf F\in \mathcal P(\mathbf R)}\sum_{\substack{{\operatorname Q'\in \mathfrak{A}(\operatorname R_{s-1}',\chi)}\\ \mathsf T\cap \mathsf T'\cong \mathbf F}}\Pb[\mathsf L_j\Join_\Gs I_\Gs(\operatorname Q')\mid \mathsf L_i\Join_\Gs I_\Gs(\operatorname Q)]\le O\big(n^\chi p_\eta^\zeta\times \max_{\mathbf F\in \mathcal P(\mathbf R)}n^{-\operatorname {Cap}(\mathbf F)}\big)\,.
\]
Since $\operatorname {Cap}(\mathbf F_\mathbf R) = \min_{\mathbf F\in \mathcal P(\mathbf R)} \operatorname {Cap}(\mathbf F)$, this yields Lemma~\ref{lem:est-for-intersecting}.
\end{proof}

\subsection{Proof of Proposition~\ref{prop:good-events-are-typical}}\label{sec:typical-good-event}

We start with some notations. For any simple graph $\mathbf H$ and $\mathbf R \subset V(\mathbf H)$ with $|\mathbf R|=r\ge 1$, and any $r$-tuple $\operatorname I \in \mathfrak{A}([n],r)$, we define $\operatorname{EXT}(\operatorname I;\mathbf R,\mathbf H)$ to be the collection of subgraphs $H$ in $  G$ satisfying the following condition: there is an isomorphism $\phi:\mathbf H\to H$ such that $\phi(\mathbf R)= ( v_i)_{i\in \operatorname I}$ (note that here the equality is equal in the sense of a tuple, not just in the sense of a set). In addition, we let  $\operatorname{Ext}(\operatorname I;\mathbf R,\mathbf H) =  |\operatorname{EXT}(\operatorname I;\mathbf R,\mathbf H)|$. For a pair $(\mathbf R,\mathbf H)$, we say it is a \emph{dense pattern}, if for any subgraph $\mathbf H'\subsetneq \mathbf H$ with $\mathbf R\subset V(\mathbf H')$,
\begin{equation}\label{eq-def-dense-pattern}
|V(\mathbf H)\setminus V(\mathbf H')|<\alpha_\eta|E(\mathbf H)\setminus E(\mathbf H')|.
\end{equation}
We say $(\mathbf R, \mathbf H)$ is a \emph{sparse pattern}, if for any subgraph $\mathbf H'\subset \mathbf H$ with $\mathbf R\subset V(\mathbf H')$ and $E(\mathbf H') \neq \emptyset$,
\begin{equation}\label{eq-def-sparse-pattern}
|V(\mathbf H')\setminus\mathbf R|>\alpha_\eta|E(\mathbf H')|. 
\end{equation}
Note that dense and sparse patterns are mutually exclusive (this can be checked by proof of contradiction with taking $\mathbf H' = \mathbf R$ in \eqref{eq-def-dense-pattern} and taking $\mathbf H' = \mathbf H$ in \eqref{eq-def-sparse-pattern}) but they are not mutually complementary, and in addition sparse pattern corresponds to the $\alpha_\eta$-safe extension in \cite{SP90}. Recall \eqref{eq:def-of-M(T)} for the definition of $\mathbb D(\mathbf T_0)$ for $\mathbf T_0\subset \mathbf T$. We introduce yet another good event $\Gc$ as follow:
\begin{defn}
	\label{def:good-event}
	Fix a large constant $\kappa_4=\kappa_4(\alpha_\eta,\mathbf T)$ which will be determined later. Define $\Gc$ as the event that  
	for any $\emptyset\neq \mathbf R\subset \mathbf L$ and any tree $\mathbf T_0 \subset \mathbf T$ with $V(\mathbf T_0)\cap \mathbf L =\mathbf R$, 
	\begin{equation}
		\label{eq:ext-bound}
		\max_{\operatorname I\in \mathfrak{A}([n],|\mathbf R|)}\operatorname {Ext}(\operatorname I;\mathbf R,\mathbf T_0)\le \kappa_4\mathbb D(\mathbf T_0)\,.
	\end{equation}
\end{defn}
\begin{lemma}
	\label{lem:G-is-typical}
	For some large constant $\kappa_4$, it holds that $\Pb[\Gc^c]=o(1)$.
\end{lemma}
\begin{proof}
	 Since the number of pairs $(\mathbf R,\mathbf T_0)$ is bounded in $n$, we only need to prove \eqref{eq:ext-bound} holds with probability tending to 1 for any fixed pair. For a fixed pair $(\mathbf R,\textbf{T}_0)$, we take a subgraph $\textbf{F}_0 \subset \mathbf T_0$ with $V(\mathbf F_0)\cap \mathbf L = \mathbf R$ such that $ \operatorname{Cap}(\textbf{F}_0)$ is minimized. Note that subgraphs in $\operatorname {EXT}(\operatorname I;\mathbf R,\mathbf T_0)$ can be ``constructed'' as follows:\\
 \emph{Step 1.} choose a subgraph $F\cong \textbf{F}_0$ in $G$ with fixed leaves in $\mathbf R$;\\
 \emph{Step 2.} add vertices and edges to $F$ and get a final subgraph $T\cong \textbf{T}$ in $G$.\\
 Since both $\mathbf{F}_0$ and $\mathbf{T}_0\setminus_\star \mathbf{F}_0$ are forests, we may assume $\mathbf{T}_1,\ldots,\mathbf{T}_l$ are the components of $\mathbf{F}_0$, and $\mathbf{T}_1',\ldots,\mathbf{T}_m'$ are the components of $\mathbf{T}_0\setminus_\star \mathbf{F}_0$. Denote $\mathbf R_i=V(\mathbf T_i)\cap \mathbf R$ for $1\le i\le l$ and $\mathbf R_j'=V(\mathbf F_0)\cap V(\mathbf T_j')$ for $1\le j\le m$. By the minimality of  $\operatorname {Cap}(\mathbf{F}_0)$ and the irrationality of $\alpha_\eta$, we see that $(\mathbf R_i,\textbf{T}_i)$ is a dense pattern for each $1\le i\le l$ and $(\mathbf R_j',\textbf{T}_j')$ is a sparse pattern for each $1\le j\le m$. It then suffices to show the following two items.
\begin{itemize}
	\item For any fixed dense pattern $(\mathbf R,\textbf{H})$, there exists a large constant $\kappa$ such that probability $1-o(1)$,
	\begin{equation}\label{eq-dense-patter-bound}
	\max_{\operatorname I\in \mathfrak A([n],|\mathbf R|)}\operatorname{Ext}(\operatorname I;\mathbf R,\textbf{H})\le \kappa\,.
	\end{equation}
	\item For any fixed sparse pattern $(\mathbf R,\textbf H)$, there exists a large constant $\kappa$ such that with probability $1-o(1)$,
	\begin{equation}\label{eq-sparse-pattern-bound}
	\max_{\operatorname I\in\mathfrak A([n],|\mathbf R|)}\operatorname {Ext}(\operatorname I;\mathbf R,\textbf{H})\le \kappa n^{|V(\textbf{H})\setminus \mathbf R|-\alpha_\eta |E(\textbf{H})|}\,.
	\end{equation}
\end{itemize}
 By \cite[Corollary 4]{SP90} (where the condition $\alpha_\eta$-safe is equivalent to the sparsity of $(\mathbf R,\mathbf H)$), we see that \eqref{eq-sparse-pattern-bound} holds. Thus, it remains to prove \eqref{eq-dense-patter-bound}. To this end, we will use the following intuition: for a dense pattern $(\mathbf R,\textbf{H})$, if $\operatorname{Ext}(\operatorname I;\mathbf R,\textbf{H})$ is large for some $\operatorname I\in \mathfrak{A}([n],|\mathbf R|)$, then there exists a subgraph $K\subset G$ with bounded size such that $|V(K)|-\alpha_\eta |E(K)|<-1$. 
We next elaborate this precisely. Since $(\mathbf R,\textbf{H})$ is dense, we have
\begin{equation}
	\label{eq:def-delta}
-\delta\stackrel{\operatorname {def}}{=}\max_{\substack{\textbf{H}'\subsetneq \textbf{H},\mathbf R\subset V(\textbf{H}')}}\big(|V(\textbf H)\setminus V(\textbf H')|-\alpha_\eta|E(\textbf H)\setminus E(\textbf H')|\big)<0\,.
\end{equation}
If $\operatorname {Ext}(\operatorname I;\mathbf R,\textbf H)$ exceeds a large integer $\kappa$ for some $\operatorname I\in \mathfrak{A}([n],|\mathbf R|)$, then there exist $\kappa$ subgraphs $H_1,\dots, H_\kappa\cong \mathbf H$ in $G$, such that each isomorphism from $\mathbf H$ to $H_i$ maps $\mathbf R$ to  $(v_i)_{i\in \operatorname I}$. Let $K_i= H_1\cup H_2\cup \cdots\cup H_i$ for $1\le i\le \kappa$. We note that whenever $K_{i+1}\setminus K_i\neq \emptyset$, each of its component is isomorphic to some $\mathbf H\setminus_\star \mathbf H'$ with $\mathbf H'\subset\mathbf H$ and $\mathbf R\subset V(\mathbf H')$. Denoting $P(K_i)=|V(K_i)|-\alpha_\eta|E(K_i)|$, we then deduce from \eqref{eq:def-delta} that 
\begin{itemize}
	\item $P(K_{i+1})\le P(K_i)$ for any $i\ge 1$, with equality holds if and only if $K_{i+1}= K_i$,
	\item If $K_{i+1}\neq K_i$, then $P(K_{i+1})\le P(K_i)-\delta$.
\end{itemize}
In addition, for each fixed $K_i$, there exists a large integer $N=N(K_i)$ depending only on the size of $K_i$ such that $K_{i+N}\neq K_i$. Combined with the fact that $P(K_1)=|V(H)|-\alpha_\eta|E(H)|$ is bounded in $n$, we see  $P(K_\kappa)<-1$ for $\kappa$ large enough. Clearly, we also have $|V(K_\kappa)|\le \kappa|V(\textbf H)|$. Therefore, 
\begin{equation}
	\label{eq:dense-pattern-union-bound}
\begin{aligned}
&\ \Pb\big[\max_{\operatorname I\in \mathfrak A([n],|\mathbf R|)}\operatorname {Ext}(\operatorname I;\mathbf R,\textbf H)\ge \kappa\big]\\
 \leq &\ \Pb\big[\exists\ {K}\subset G \text{ with }|V(K)|\le \kappa|V(\mathbf H)|\text{ and }P(K)<-1\big]\\
\leq & \sum_{\substack{|V(\textbf K)|\le \kappa|V(\textbf H)|\\ P(\mathbf{K})<-1}} n^{|V(\textbf K)|}p_\eta^{|E(\textbf K)|}=\sum_{\substack{|V(\textbf K)|\le \kappa|V(\textbf H)|\\ P(\mathbf{K})<-1}}  n^{P(\textbf K)}=O\big(n^{-1}\big)\,.
\end{aligned}
\end{equation}
This proves \eqref{eq-dense-patter-bound}, and thus completes the proof of the lemma.
\end{proof}
 
We will prove Proposition~\ref{prop:good-events-are-typical} by induction. Assuming for some fixed $1\le s\le S$ it holds uniformly that for all $1\leq t\leq s-1$,
\begin{equation}\label{eq-G-good-induction-assumption}
\Pb\big[(\Gc_t^1)^c\big]+\Pb\big[(\Gc_t^2)^c\big]=o(1)\,,
\end{equation} 
we will show that \eqref{eq-G-good-induction-assumption} holds for $t=s$ (This will in particular prove the base case of $t=1$). Recall definition \eqref{def:X_s} for $X_t,1\le t\le s-1$, the induction hypothesis \eqref{eq-G-good-induction-assumption} is used in the following lemma:
 	\begin{lemma}
 		\label{lem:good-steps}
 		Assuming \eqref{eq-G-good-induction-assumption} for $1\leq t\leq s-1$, we have that
 		\begin{equation}
 		\label{eq;good-step-est}
 		\Pb[X_t<\log n]=o(1)\text{ uniformly for all }1\le t\le s-1\,.
 		\end{equation}
 	\end{lemma}
 \begin{proof}
 	From the induction hypothesis we see $\Pb[X_t<\log n]$ is no more than $$ \Pb\big[(\Gc_t^1)^c\big]+\Pb\big[(\Gc_t^2)^c\big]+\Pb\big[X_t<\log n\mid \Gc_t^1\cap \Gc_t^2\big]=o(1)+\Pb\big[X_t<\log n\mid \Gc_t^1\cap \Gc_t^2\big]\,.$$
 	Note that $\mathbb E[X_t\mid \Gc_t^1\cap \Gc_t^2]\gg \log n$, by applying the Paley-Zygmund inequality and then make use of Proposition~\ref{prop:two-point-est}, we get the second term above is also $o(1)$. The estimates are uniform for all $1\le t\le s-1$ and the result follows. 
 \end{proof}
To show the desired result, we cannot condition on the full information of $\mathcal{F}_{s-1}$. Instead, we turn to $\mathcal{I}_{s-1}=\sigma\{\operatorname M_{s-1}, \pi(\operatorname M_{s-1}),\operatorname{MT}_t,1\le t\le s-1\}$. To the contrary of $\mathcal{F}_{s-1}$, the information of $\bigcup_{1\le t\le s-1}\operatorname {EXP}_t$ is not fully contained in $\mathcal I_{s-1}$. Recall the set $\mathtt M_s$ in Algorithm~\ref{algo:greedy} and note that $\mathtt M_s$ can be determined from $\mathcal I_{s-1}$. In addition, by our choice $\kappa_0>4\zeta/\eta$ and the fact that each $\operatorname{MT}_t$ contains no more than $2\zeta$ elements in $[n]$, it holds deterministically that $|\mathtt{M}_s|\ge {\eta n}/2$. Denote $V_{R}=\{v_i:i\in \operatorname R_{s-1}\}$ (recall $\operatorname R_{s-1} = [n]\setminus \operatorname M_{s-1}$) and $  V_{\mathtt M}=\{v_i:i\in \mathtt M_{s}\}$. For any $v\in V$, let $N_{R}(v)$ and $N_{\mathtt M}(v)$ be the numbers of neighbors of $v$ in $V_{R}$ and $V_{\mathtt M}$, respectively. The next lemma describes the properties of the conditional distribution of $G$ given $\mathcal I_{s-1}$, which will be useful later.
\begin{lemma}
	\label{lem:dominated-by-ER}
	Recall $E_0$ as in \eqref{def:E0 and Es0}. For any realization of  $\mathcal I_{s-1}$, the graph $G$ conditioned on $\mathcal I_{s-1}$ is given by $ {GO}_{s-1}\cup  {GO}^\dagger_{s-1}$, where $ {GO}^\dagger_{s-1}$ is a graph on $V$ with edges in $  E_0\setminus E( {GO}_{s-1})$ which
	is stochastically dominated by an \ER graph on $(V,  E_0\setminus E({GO}_{s-1}))$ with edge density $p_\eta$ (i.e., each edge in $E_0\setminus E({GO}_{s-1})$ is preserved with probability $p_\eta$).
\end{lemma}
\begin{proof}
	It is clear that $ {GO}_{s-1} \subset G$ under such conditioning, and thus it remains to understand the behavior for the remaining graph ${GO}^\dagger_{s-1}$. 
	For any $e \in E_0\setminus E( {GO}_{s-1})$ and any realization $\omega_{\setminus  e}$ for edges in $ E_0\setminus E( {GO}_{s-1})$ except $e$, note that if $\omega_{\setminus  e}$ together with $e\in G$ (as well as $ {GO}_{s-1} \subset E$) yields the realization $\mathcal I_{s-1}$, then $\omega_{\setminus e}$ together with $e\not\in E$ also yields the realization $\mathcal I_{s-1}$. Therefore, 
	$\Pb[e\not\in E \mid \mathcal I_{s-1}, \omega_{\setminus e}] \geq 1 - p_\eta$, completing the proof.
\end{proof}
In what follows, we fix a realization of $\mathcal I_{s-1}$. 
\begin{defn}
\label{def: attach graph}
	For a triple $(\mathbf R,\mathbf H,\mathbf v)$ with $\mathbf R\subset V(\mathbf H), \mathbf v\in V(\mathbf H)\setminus \mathbf R$ and a vertex $v\in V_{R} =\{v_i:i\in [n]\setminus \operatorname M_{s-1}\}$, we say a subgraph $H \subset G$ is an $(\mathbf R,\mathbf H,\mathbf v)$-\emph{attaching graph rooted at} $v$ if there is an isomorphism $\phi:\mathbf H\to H$, such that $\phi(\mathbf R)\subset M_{s-1},\phi(\mathbf v)= v$, and any two vertices in $\phi(\mathbf R)$ has graph distance at most $\zeta$ on $ {GO}_{s-1}$.
\end{defn}
\begin{lemma}
	\label{lem:last-lemma}
	There exists a large constant $\kappa=\kappa(\eta,\alpha_\eta,\mathbf T,\kappa_0)$ such that  $\Pb[\mathcal G_\kappa|\mathcal I_{s-1}] = 1-o(1)$ where $\mathcal G_{\kappa} = \mathcal G_{\kappa, s}$ is the following event: for any triple $(\mathbf R,\mathbf H,\mathbf v)$ where $(\mathbf R,\mathbf H)$ is a dense pattern with $\mathbf H$ being a subtree of $\mathbf T$ and $\mathbf R=V(\mathbf H)\cap \mathbf L$, and for any vertex $  v\in   V_{  R}$, the number of $(\mathbf R,\mathbf H,\mathbf v)$-attaching graphs rooted at $v$ is bounded by $\kappa$.
\end{lemma}
\begin{proof}
	The proof is similar to  that of Lemma~\ref{lem:G-is-typical}, and we begin with introducing $P_\zeta(K)$ as an analogue of $P(K)$ in the proof of Lemma~\ref{lem:G-is-typical}. For a subgraph $K\subset G$, we draw an adjunctive edge between any pair of vertices $u, v\in V(K)$ if and only if $u, v$ has graph distance at most $\zeta$ on $ {GO}_{s-1}$. This gives an adjunctive graph on $V(K)$, and we denote the collection of its components by $\mathfrak C_\zeta(K)$. Then we define $P_\zeta(K)=|\mathfrak C_\zeta(K)|-\alpha_\eta|E(K)\setminus E(GO_{s-1})|$.  
	
Our proof is essentially by contradiction, that is, we will show that if $\mathcal G_\kappa$ fails then a rare event must occur. To this end, let $H_1,\dots, H_\kappa\cong \mathbf H$ be distinct $(\mathbf R,\mathbf H,\mathbf v)$-attaching graphs rooted at some $v\in V_{R}$,  let $\phi_i:\mathbf H\to H_i$ be the isomorphism as in Definition~\ref{def: attach graph} and let $ K_{i}= H_1\cup \cdots\cup H_i$ for $1\le i\le \kappa$. We claim that for $1\leq i < \kappa$,
\begin{equation}\label{eq-P-zeta-monotone}
P_\zeta(K_{i+1}) \leq P_\zeta(K_{i}) - \delta \mathbf 1_{\{K_{i+1}\not\subset K_i\cup {GO}_{s-1}\}}\,,
\end{equation}
where $\delta>0$ is a constant which does not depend on $i$ or $\kappa$. We now fix $i$ and prove \eqref{eq-P-zeta-monotone}. To this end, we consider components $C_1,\dots, {C}_r$ of  $K_{i+1}\setminus_\star K_i$ (recall the definition of $H_1\setminus_\star H_2$ for two simple graphs as in the proof of Lemma~\ref{lem:designing-tree}, and we view $  C_j,1\le j\le r$ as subgraphs of $H_{i+1}$). Let $N_j$ be the number of components in $\mathfrak C_\zeta( {K}_{i+1})$ intersecting $ {C}_j$ but \emph{not} containing $v$, and let $E_j = |E(C_j) \setminus E(GO_{s-1})|$. Then, we have
\[
P_\zeta(K_{i+1})-P_\zeta(K_i)\le\sum_{j=1}^r (N_j-\alpha_\eta E_j)\,.
\]
For each ${C}_j$, let ${F}_j$ be the subgraph on vertices  $V(H_{i+1})\setminus\left(V( {C}_j) \setminus (V( {K}_i)\cup V({GO}_{s-1}))\right)$ with edges $E(H_{i+1})\setminus(E({C}_j) \setminus (E(K_i)\cup E({GO}_{s-1})))$. We further write $F_j= {F}_{j,0}\cup  {T}_{j, 1}\cup\ldots {T}_{j, k_j}$, where $ {F}_{j, 0}$ is the union of components of $ {F}_j$ which intersect  $\phi_{i+1}(\mathbf R)$, and $ {T}_{j, l}$'s (for $l=1,\dots,k_j$) are the remaining tree components (here possibly $k_j = 0$). Writing $\mathbf F_{j,0} =\phi_{i+1}^{-1}( {F}_{j,0})$, we have
\begin{equation}\label{eq-equality-E-j}
E_j=|E( {H}_{i+1})|-|E( {F}_{j,0})|-\sum_{l=1}^{k_j}|E( {T}_{j, l})|=|E(\mathbf H)\setminus E(\mathbf F_{j,0})|-\sum_{l=1}^{k_j} |E( {T}_{j, l})|\,.
\end{equation}
 For the estimation of $N_j$, we claim that 
 \begin{equation}
 	\label{eq:est-Nj}
 	N_j\le k_j + |V( {H}_{i+1})|-|V( {F}_{j,0})|-\sum_{l=1}^{k_j}|V( {T}_{j, l})|=|V(\mathbf H)\setminus V(\mathbf F_{j,0})|-\sum_{l=1}^{k_j}|E( {T}_{j,l})|\,.
 \end{equation}
Indeed, since $V( {T}_{j, l})$ belongs to a single component in  $\mathfrak C_\zeta( {K}_{i+1})$ for $1\le l\le k_j$ and also the vertices in $ {F}_{j,0}$ are in a single component  in  $\mathfrak C_\zeta( {K}_{i+1})$ (by the definition of $(\mathbf R,\mathbf H,\mathbf v)$-attaching and the fact that $\phi_{i+1}(\mathbf R)\subset V( {F}_{j,0})$), we get that the number of components in $\mathfrak{C}_\zeta( {K}_{i+1})$ which intersect $ {C}_j$ is at most
$
1+k_j+|V( {H}_{i+1})|-|V( {F}_{j,0})|-\sum_{l=1}^{k_j}|V( {T}_{j,l})|
$ (this is because, the number of components on $V(F_j)$ is at most $1+k_j$ and each vertex outside $V(  F_j)$ induces at most one component).
Moreover, since $N_j$ does not count the component in $\mathfrak{C}_\zeta( {K}_{i+1})$ which contains $v$, we can remove one of the above components (possibly $F_{j,0}$ or one of the $T_{j,l}$'s) when counting $N_j$. This verifies \eqref{eq:est-Nj}.
Combined with \eqref{eq-equality-E-j}, it yields that
\begin{align*}
N_j-\alpha_\eta E_j
\le &\ |V(\mathbf H)\setminus V(\mathbf F_{j,0})|-\alpha_\eta |E(\mathbf H)\setminus E(\mathbf F_{j,0})|-(1-\alpha_\eta)\sum_{l=1}^{k_j}|E(T_{j, l})|\,,
\end{align*}
which is $\leq -\delta\mathbf{1}_{\{\mathbf{F}_{j}\neq \mathbf{H}\}}$ for some $\delta>0$ since $(\mathbf R,\mathbf H)$ is a dense pattern (recall \eqref{eq:def-delta} and $\mathbf R\subset V(\mathbf F_{j,0})$). Letting $\mathbf F_j = \phi_{i+1}^{-1}(F_j)$, we then conclude \eqref{eq-P-zeta-monotone} by the observation that $$\{{K}_{i+1}\not\subset {K}_i\cup {GO}_{s-1}\}\subset \bigcup_{j=1}^r \{\mathbf{F}_j\neq \mathbf H\}, \mbox{ and thus } \mathbf{1}_{\{{K}_{i+1}\not\subset {K}_i\cup {GO}_{s-1}\}}\le \sum_{j=1}^r\mathbf{1}_{\{\mathbf F_j\neq \mathbf H\}}\,.$$

In addition, for each $i$, we claim that there exists a large constant $N$ depending only on the size of $K_i$, such that $K_{j+1}\subset K_j\cup GO_{s-1}$ cannot hold simultaneously for all $j\in\{i,i+1,\dots,i+N\}$. We prove this by contradiction. Suppose otherwise and then we see $H_{i+1},\dots,H_{i+N}$ must all be contained in $K_i\cup GO_{s-1}$. Furthermore, since each $H_j$ is connected, these graphs must be contained in the $\zeta$-neighborhood of $v$ in $K_i\cup GO_{s-1}$. Note that the number of vertices in this $\zeta$-neighborhood is at most $\sum_{j=0}^\zeta(\Delta(GO_{s-1}))^j$ where $\Delta(GO_{s-1})$ is the maximal degree and is uniformly bounded. Recalling \eqref{eq-degree-bound}, we arrive at a contradiction if $N$ is chosen sufficiently large.

Combining the preceding claim and \eqref{eq-P-zeta-monotone}, we can choose $\kappa= \kappa(\zeta, N, \delta, \chi)$ sufficiently large such that that $P_\zeta(K_\kappa) < -1$ on $\mathcal G_\kappa^c$. We next bound the enumeration for $K_\kappa$ given $\mathcal I_{s-1}$. To this end, we note that for each component in $\mathfrak C_\zeta(K_\kappa)$ the number of choices is $O(n)$ (where the $O$-term depends on $(\zeta, N, \delta, \chi)$); this is because for each such component once a vertex is fixed the number of choices for remaining vertices is $O(1)$ by connectivity and by \eqref{eq-degree-bound}.  Then in light of Lemma~\ref{lem:dominated-by-ER} (i), we can show that $\Pb[\mathcal G_\kappa^c] \to 0$ via a union bound over all possible choices of $K_\kappa$ (which is similar to that for  \eqref{eq:dense-pattern-union-bound}). This completes the proof.
\end{proof}

Now we are ready to present the proof of Proposition~\ref{prop:good-events-are-typical}.
\begin{proof}[Proof of Proposition~\ref{prop:good-events-are-typical}]
	Assume \eqref{eq-G-good-induction-assumption}. Recall $\Gc$ as in Definition~\ref{def:good-event} and recall $\mathcal G_\kappa = \mathcal G_{\kappa, s}$ from the statement of Lemma~\ref{lem:last-lemma}. For $\Gc_s^1$, recall Definition~\ref{def:G-s-1} and the notations therein. We claim that $\Gc\cap \Gc_{\kappa}\subset \Gc_s^1$. Provided with this, we obtain from Lemmas~\ref{lem:G-is-typical} and \ref{lem:last-lemma} that
	\[
	\Pb\big[(\Gc_s^1)^c\big]\le \Pb\big[\Gc^c\big]+\Pb\big[(\Gc_\kappa)^c\big]=o(1)+\mathbb{E}\Big[\Pb\big[(\Gc_\kappa)^c\mid \mathcal I_{s-1}\big]\Big]=o(1)\,.
	\]
We next prove the claim. To this end, for each fixed pair $(\mathbf F_1,\mathbf F_2)=(\mathbf T_0,\mathbf T_1\cup \cdots\cup \mathbf T_l)\in \mathcal P_0$, denote $\mathbf R_i=V(\mathbf T_i)\cap \mathbf L$ for $0\le i\le l$. From the definition of $\mathcal P_0$, we see $\mathbf T_i$ is a dense tree and thus $(\mathbf R_i,\mathbf T_i)$ is a dense pattern for each $1\le i\le l$. 

Fix an arbitrary subgraph $\mathsf T\cong \mathbf T$ in $\mathsf G$ with leaf set $\mathsf L$ and $V(\mathsf T) \cap \mathsf{M}_{s-1} = \mathsf L$. Recall  the definition of $\operatorname{Enum}(\mathbf F_1,\mathbf F_2)$ (with respect to $\mathsf T$) below \eqref{eq:containing-relation}: it counts the number of $\xi$-tuples $\operatorname L'\in \bigcup_{1\le t \le s-1}\operatorname{EXP}_{t}$ which satisfy that there are two embeddings $\phi_1:\mathbf T_0\to \mathsf T\cup \mathsf{GO}_{s-1}$ and $ \phi_2:\mathbf T_1\cup \cdots\cup \mathbf T_l\to \mathsf T\cup\mathsf{GO}_{s-1}$ such that $$\phi_1\big(E(\mathbf T_0)\big)\cap E(\mathsf T)\neq \emptyset\text{ and }\phi_1(\mathbf R_0)\cup\phi_2(\mathbf R_1\cup\cdots\cup \mathbf R_l)= I_G(\operatorname L')\,.$$ 
	Note that $GO_{s-1}\cong \mathsf{GO}_{s-1}$ through $\Pi\circ I_G^{-1}$. Combined with the fact that $\operatorname{L}'\in \bigcup_{1\le t\le s-1}\operatorname{EXP}_t$ implies $\operatorname{L}'\Join_{G} [n]$, 
	it yields that each such tuple corresponds to an embedding $\psi:\mathbf T\to G$ which satisfies the following conditions: \\
	\noindent(i) $\psi(\mathbf R_0)$ is contained in the $\zeta$-neighborhood of $L=I_G\circ \Pi^{-1}(\mathsf L)$ on $GO_{s-1}$; \\
	\noindent(ii) for each $1\le i\le l$, we have $\psi(\mathbf R_i)\subset GO_{s-1}$ and the diameter of $\psi(\mathbf R_i)$ with respect to the graph metric  on ${GO}_{s-1}$ is at most $\zeta$. 
	
	Thus, we can bound $\operatorname{Enum}(\mathbf F_1,\mathbf F_2)$ by the number of such embeddings. To this end, note that on $\Gc$ we have 
	\begin{equation}\label{eq-moderate-degree-condition}
	\Delta(G)\le \kappa_4np_\eta\,.
	\end{equation}
For the number of possible realizations of $\psi\big(V(\mathbf T_0)\big)$, when $\mathbf R_0=\emptyset$ it is bounded by $O\big(n\times (np_\eta)^{|E(\mathbf T_0)|}\big)$ using \eqref{eq-moderate-degree-condition}; when $\mathbf R_0\neq \emptyset$ it is bounded by $O\big(\mathbb D(\mathbf T_0)\big)$ using \eqref{eq-degree-bound} and \eqref{eq:ext-bound} (More precisely, the number of ways of choosing $\psi(\mathbf R_0)$ is $O(1)$ by \eqref{eq-degree-bound} and for each fixed $\psi(\mathbf R_0)$, the number of ways of choosing the rest of $\psi\big(V(\mathbf T_0)\big)$ is $O\big(\mathbb{D}(\mathbf T_0)\big)$ by \eqref{eq:ext-bound}).
In addition, for any edge $(\mathbf u,\mathbf v)\in E(\mathbf T)\setminus( E(\mathbf T_0)\cup \cdots \cup E(\mathbf T_l))$, once $\psi(\mathbf u)$ is fixed, the number of choices for $\psi(\mathbf v)$ is $O(np_\eta)$ by \eqref{eq-moderate-degree-condition}.
		For each $1\le i\le l$ and any $\mathbf v\in V(\mathbf T_i)$, once $\psi(\mathbf v)$ is fixed, each realization of $\psi\big(V(\mathbf T_i)\big)$ corresponds to a $(\mathbf R_i,\mathbf T_i,\mathbf v)$-attaching graph rooted at $\psi(\mathbf v)$, and thus the number of such realizations is at most $\kappa$ by $\Gc_\kappa$.

Provided with preceding observations, we may bound the number of these embeddings (which in turn bounds $\operatorname{Enum}(\mathbf F_1,\mathbf F_2)$) by first choosing $\psi\big(V(\mathbf T_0)\big)$ and then choosing the $\psi$-values for the remaining vertices on $\mathbf T$ inductively. More precisely, the ordering for choosing vertices satisfy the following properties (see Figure~\ref{fig:enumeration_bound} for an illustration):
(i) we first determine the $\psi$-value of vertices in $V(\mathbf T_0)$;
(ii) for the remaining vertices, the vertex whose $\psi$-value is to be chosen is neighboring to some vertex whose $\psi$-value has been chosen;
(iii) whenever we choose $\psi(\mathbf v)$ for any $\mathbf v\in V(\mathbf T_i)$, we also choose $\psi\big(V(\mathbf T_i)\big)$ immediately (which has at most $\kappa$ choices as noted above). Therefore, we conclude that 
\[
\operatorname{Enum}(\mathbf F_1,\mathbf F_2)\lesssim \begin{cases}
n\times (np_\eta)^{|E(\mathbf T_0)|}\times (np_\eta)^{|E(\mathbf T)\setminus (E(\mathbf T_0\cup \cdots\cup (\mathbf T_l))|},\quad&\text{if }\mathbf R_0=\emptyset\,,\\
\mathbb D(\mathbf T_0)\times (np_\eta)^{|E(\mathbf T)\setminus (E(\mathbf T_0\cup \cdots\cup (\mathbf T_l))|},\quad& \text{if }\mathbf R_0\neq \emptyset\,,
\end{cases}
\]
\begin{figure}[ht]
\centering
\includegraphics[]{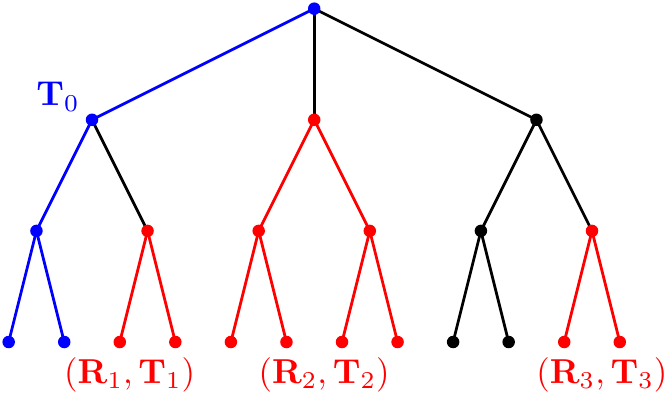}
\caption{In order to enumerate, as described above we choose the blue, black and red parts of $\mathbf T$ in order.}
\label{fig:enumeration_bound}
\end{figure}
which implies \eqref{eq:good-s-1} with an appropriate choice of $\kappa_1$. This proves the claim.

	Next we investigate $\Gc_s^2$ as in Definition~\ref{def:G-s-2}. For Item (i), we first note that by the sub-Gaussian concentration of Bernoulli variables, it holds that each vertex in $G$ has degree at least $(1-\eta/8)np_\eta$ except with exponentially small probability. Furthermore, from Lemma~\ref{lem:dominated-by-ER} we also obtain that under any conditioning of $\mathcal I_{s-1}$, except with exponentially small probability we have for any $v\in V_R$, the total number of edges between $v$ and vertices in $V\setminus V_{\mathtt M}$ is no more than $(1+\eta/8)(n-|\mathtt{M}_s|)p_\eta$. Therefore, with probability $1-o(1)$ for each vertex $v\in V_R$ the number of neighbors in $V_{\mathtt M}$ is at least
	\[
	(1-\eta/8)np_\eta-(1+\eta/8)(n-|\mathtt M_s|)p_\eta\ge( |\mathtt M_s|-\eta n/4)p_\eta\ge \eta np_\eta/4\,.
	\]
 A similar argument applies for $V_R$ and we get that with probability $1-o(1)$, each $v\in V_R$ has at least $\eta np_\eta/4$ neighbors in $V_R$.

	Provided with these degree lower bounds, the number of $\xi$-tuples in $\operatorname{CAND}_s$ 
	is at least of order $(np_\eta)^{\zeta}$. In order to get $\operatorname{GC}_s$, we still need to remove those $s$-bad tuples in $\operatorname{CAND}_s$, and it suffice to show that the total number of removed tuples is typically $o\big((np_\eta)^\zeta\big)$. By Markov's inequality, we only need to show that $\mathbb{E}| \operatorname{CAND}_s\setminus \operatorname{GC}_s|= o\big((np_\eta)^\xi\big)$. Taking conditional expectation with respect to $\mathcal I_{s-1}$ and using the averaging over conditioning principle, we see
	\begin{equation}
		\label{eq:exp-of-removed-tuples}
		\mathbb{E}|\operatorname{CAND}_s\setminus \operatorname{GC}_s|= \mathbb E\Big[\sum_{\operatorname{L}\in\mathfrak{A}(\mathtt{M}_s,\xi)} \Pb\big[\operatorname{L}\in \operatorname{CAND}_s\setminus \operatorname{GC}_s\mid \mathcal I_{s-1}\big]\Big]\,.
	\end{equation}
Denote $\mathfrak{A}_s$ as the set of tuples in $\mathfrak{A}(R_{s-1},\chi)$ with the first coordinate equals to $u_s$ (which is the minimal number in $R_{s-1}$), from the definition of $\operatorname{CAND}_s$ and a union bound we see 
    \begin{equation}
    \label{eq:last-bound-1}
   \Pb\big[\operatorname{L}\in \operatorname{CAND}_s\setminus\operatorname{GC}_s\mid \mathcal I_{s-1}\big]
   \leq \sum_{\operatorname{Q}\in \mathfrak{A}_s}\Pb\big[I_G(L)\Join_G Q,\operatorname{L}\notin \operatorname{GC}_s\mid \mathcal I_{s-1}\big]\,,
    \end{equation}
    where $L,Q$ denote for $I_G(\operatorname{L}),I_G(\operatorname{Q})$, respectively. We further denote $T_{L,Q}$ for the tree that certifies the event $\{L\Join_{G} Q\}$.
    From Definition~\ref{def:s-good}, $\operatorname{L}\notin \operatorname{GC}_s$ implies that there exists a subgraph $T^*\cong \mathbf T$ in $T_{L,Q}\cup GO_{s-1}$ with leaf set $L^*$, such that $E(T^*)\cap E(T_{L,Q})\neq \emptyset$ and $I_G^{-1}(L^*)\in \operatorname{EXP}_{t^*}$ for some $1\le t^*\le s-1$. 
    We write $\operatorname{d}_{GO_{s-1}}$ for the graph distance on $GO_{s-1}$, then it is readily to see that the two tuples $\operatorname{L}=I_G^{-1}(L)=(l_1,\dots,l_\xi)$ and $\operatorname{L}^*=(l_1^*,\dots,l_\xi^*)$ must satisfy the following two properties:\\
    \noindent (i) For any $i\in [\xi]$, there exists $j\in [\xi]$ such that $\operatorname{d}_{GO_{s-1}}(v_{l_i^*},v_{l_j})\le \zeta$.\\
    \noindent (ii) For at least two indices $p\in [\xi]$, there exists $q\in [\xi]$ such that $\operatorname{d}_{GO_{s-1}}(v_{l_p},v_{l_q^*})\le \zeta$. \\
    We denote by $\mathtt L(\operatorname{L})$ the collection of all $\xi$-tuples $\operatorname{L}^*\in \mathfrak{A}(\operatorname{M}_{s-1},\xi)$ that satisfy Property (i) above.

    For any triple $(\operatorname{L},\operatorname{Q},\operatorname{L}^*)$ with $\operatorname{L}^*\in \mathtt L(\operatorname{L})$, we divide the event $\operatorname{L}^*\in \operatorname{EXP}_{t^*}$ for some $1\le t^*\le s-1$ into two cases according to whether 
    \begin{equation}\label{eq-condition-t}
    	t^* \in \mathtt t(\operatorname{L}) \mbox{ where } \mathtt t(\operatorname{L}) = \{t': d_{GO_{s-1}}(v_{u_{t'}},v_{l_i})\le \zeta\text{ for some }i\in [\xi]\}\,.
    \end{equation}
   For the case that $t^*\notin \mathtt t(\operatorname{L})$, the tree that certifies $\operatorname{L}'\in \operatorname{CAND}_{t^*}$ can not be fully contained in $T_{L,Q}\cup GO_{s-1}$. 
   From Lemma~\ref{lem:dominated-by-ER} and the proof of Lemma~\ref{lem:unique-is-whp} (see e.g., \eqref{eq-unique-is-typical}), we see such case happens with probability $o(1)$ uniformly under any conditioning $\mathcal I_{s-1}\cup \{T_{L,Q}\subset G\}$. For the case that $t^*\in \mathtt t(\operatorname{L})$, we apply a union bound and thus we get
	\begin{align*}
	\Pb[L\Join_{G} Q,\operatorname{L}\notin \operatorname{GC}_s\mid \mathcal{I}_{s-1}]\le&\  o(1)\times \Pb[L\Join_{G} Q\mid \mathcal I_{s-1}]\\
	+&\ \sum_{\operatorname{L}^*\in \mathtt L(\operatorname{L})}\sum_{t^* \in \mathtt t(\operatorname{L})}\Pb[L\Join_{G} Q,\operatorname{L}^*\in \operatorname{EXP}_{t^*}\mid \mathcal I_{s-1}]\,.
	\end{align*}
    From Lemma~\ref{lem:dominated-by-ER} we see $\Pb[L\Join_G Q\mid \mathcal I_{s-1}]\le p_\eta^\zeta$ for any realization of $\mathcal I_{s-1}$. Thus the first term above is $o(p_\eta^\zeta)$. To treat the remaining terms, we note that\\
    \noindent (a) $\operatorname{L}^*\in \operatorname{EXP}_{t^*}$ implies that $ \operatorname{L}^*\in \bigcup_{1\le t\le s-1}\operatorname{SUC}_{t}$ or $\operatorname{L}^*\in \operatorname{FAIL}_{t^*}\setminus\bigcup_{1\le t\le s-1}\operatorname{SUC}_{t}$; \\
    \noindent (b) if $\operatorname{L}^*\in\bigcup_{1\le t\le s-1} \operatorname{SUC}_{t}$, then from Property (ii) above,
    \begin{equation}\label{eq-L-condition-b}
   \mbox{ there are two indecies $i,j\in \operatorname{L}$ such that $\operatorname{d}_{GO_{s-1}}(v_{i},v_{j})\le 3\zeta$};\\
   \end{equation}
    \noindent (c) $\{\operatorname{L}^*\in \operatorname{FAIL}_{t^*}\} = \{I_{t^*}=0\} \cup \{I_{t^*}=1, \operatorname{L}^*\prec \operatorname{L}_{t^*}\}$ (recall that $\operatorname{MT}_t=(\operatorname{L}_t,\operatorname{Q}_t,\operatorname{Q}_t')$ for $1\le t\le s-1$ with $I_t=1$).\\
    Combining these together, we see for each pair $(\operatorname{L},\operatorname{Q})$, any $\operatorname{L}^*\in \mathtt L(\operatorname{L})$ and any $t^*\in \mathtt t(\operatorname{L})$, it holds that $\Pb[L\Join_{G} Q,\operatorname{L}^*\in \operatorname{EXP}_{t^*}\mid \mathcal I_{s-1}]$ is bounded by $p_\eta^\zeta$ times 
\begin{equation}
	\label{eq:bound-three-terms}
    \begin{aligned}
    &\ \mathbf{1}_{\{\operatorname{L}\text{ satisfies }\eqref{eq-L-condition-b}\}}
    +\Pb[X_t<\log n\mid \mathcal I_{s-1},L\Join_G Q]\\
    +&\ \Pb[X_{t^*}\ge \log n,\operatorname{L}^*\notin \bigcup_{1\le t\le s-1}\operatorname{SUC}_{t},\operatorname{L}^*\prec \operatorname{L}_{t^*}\mid \mathcal I_{s-1},L\Join_{G} Q]\,.
    \end{aligned}
\end{equation}
   In order to bound the second term in \eqref{eq:bound-three-terms}, note that if $\{X_t<\log n\}$ holds under $\mathcal I_{s-1}\cap \{L\Join_G Q\}$, then it also holds under $\mathcal I_{s-1}$ together with any other configuration on the set $E(T_{L, Q})$. Thus,
    \[
    \Pb[X_{t^*}<\log n\mid \mathcal I_{s-1},L\Join_{G} Q]\le \Pb[X_{t^*}<\log n\mid \mathcal I_{s-1}]\,.
    \]
   In order to bound the third term in \eqref{eq:bound-three-terms}, denote $\operatorname{Avai}_t$ for the (random) set of tuples in $\operatorname{CAND}_t$ that can be successfully matched, $1\le t\le s-1$. Then $|\operatorname{Avai}_t|=X_t$ and the conditional law of $\prec$ given $\mathcal I_{s-1}\cap \{L\Join_{G} Q\}$ is uniform conditioned on the following event:
    \[
    \{\operatorname{L}_t\text{ is $\prec$-minimal among }\operatorname{Avai}_t, \mbox{ for all } 1\le t\le s-1 \text{ with }I_t=1\}\,.
    \]
    In particular, under such conditioning, for $\operatorname{L}^*\notin \bigcup_{1\le t\le s-1}\operatorname{SUC}_{t}$ it holds with conditional probability $1/X_{t^*}$ that $\operatorname{L}^* \prec \operatorname{L}_{t^*}$. As a result, the third term is bounded by $(\log n)^{-1}=o(1)$.
    
    Summing over all $(\operatorname{L},\operatorname{Q})$ and combining all the arguments above altogether, we get 
    \begin{align}\label{eq-big-sum-expression}
    \sum_{\operatorname{L}\in\mathfrak{A}(\mathtt{M}_s,\xi)} &\Pb\big[\operatorname{L}\in \operatorname{CAND}_s\setminus \operatorname{GC}_s\mid \mathcal I_{s-1}\big]  \nonumber\\
    \leq
    \sum_{\operatorname{L}\in \mathfrak{A}(\mathtt{M}_s,\xi)}\sum_{\operatorname{Q}\in \mathfrak{A}_s}\sum_{\operatorname{L}^*\in \mathtt L(\operatorname{L})}\sum_{t^*\in \mathtt t(\operatorname{L})}&p_\eta^\zeta\Big(o(1)+\mathbf{1}_{\{\operatorname{L}\text{ satisfies }\eqref{eq-L-condition-b}\}}+\Pb[X_{t^*}<\log n\mid \mathcal I_{s-1}]\Big)\,.
    \end{align}
   We now bound the number of effective terms in the summation of \eqref{eq-big-sum-expression}. Clearly, $|\mathfrak{A}(\mathtt{M}_t,\xi)|\le n^\xi$ and $|\mathfrak{A}_s|\le n^{\chi-1}$. In addition, by \eqref{eq-degree-bound} we see both $|\mathtt L(\operatorname{L})|$ and $|\mathtt t(\operatorname{L})|$ are uniformly bounded for $\operatorname{L}\in \mathfrak{A}(\mathtt{M}_s,\xi)$. Furthermore, the number of $\operatorname{L}\in \mathfrak{A}(\mathtt{M}_s,\xi)$ that satisfy \eqref{eq-L-condition-b} is of order $O\big(n^{\xi-1}\big)$, and for each $1\le t^*\le s-1$ the number of tuples $\operatorname{L}\in \mathfrak{A}(\mathtt M_s,\xi)$ such that $\mathtt t(\operatorname{L})$ contains $t^*$ is also of order $O\big(n^{\xi-1}\big)$ (since some vertex in $L$ must be close to $v_{u_t}$ on the graph $GO_{s-1}$). Altogether, we see the right hand side of \eqref{eq-big-sum-expression} is upper-bounded by (recall that $\zeta=\xi+\chi-1$)
    \[
    O\big(n^\zeta\big)\times o\big(p_\eta^\zeta\big)+O\big(n^{\zeta-1}\big)\times p_\eta^\zeta+O\big(n^{\zeta-1}\big)\times p_\eta^\zeta\times\sum_{1\le t^*\le s-1}\Pb[X_{t^*}<\log n\mid \mathcal I_{s-1}]\,.
    \]
  Averaging over $\mathcal I_{s-1}$ and recalling \eqref{eq:exp-of-removed-tuples} and \eqref{eq-big-sum-expression}, we obtain that
    \[\mathbb E|\operatorname{CAND}_s\setminus \operatorname{GC}_s| = 
    o\big((np_\eta)^\zeta\big)+O\big(n^{\zeta-1}p_\eta^\zeta\big)\sum_{1\le t^*\le s-1}\Pb[X_{t^*}<\log n]\stackrel{\eqref{eq;good-step-est}}{=}o\big((np_\eta)^\zeta\big)\,,
    \]
	as desired. This implies that Item (i) of $\mathcal G_s^2$ in Definition~\ref{def:G-s-2}  holds with probability $1-o(1)$.

Finally, we treat Item (ii) of $\mathcal G_s^2$ in Definition~\ref{def:G-s-2} and it suffices to show that it holds on the event $\Gc$. Since each tuple in $\operatorname{GC}_s$ corresponds to a subgraph $T\cong \mathbf T$ in $G$ rooted at $v_{u_s}$, \eqref{eq-moderate-degree-condition} implies that  $|\operatorname{GC}_s|=O\big((np_\eta)^\zeta\big)$. For each fixed $\operatorname L_i\in \operatorname{GC}_s$ and a nonempty subset $\mathbf R \subset \mathbf L$, we bound the number of $\operatorname L_j\in \operatorname{GC}_s$ such that $(\operatorname L_i,\operatorname L_j)\in \operatorname{IP}_s(\mathbf R)$ as follows: each such $\operatorname L_j$ corresponds to a subgraph ${T}\cong \mathbf T$ of ${G}$ with leaves in $\mathbf R$ mapped to a fixed subset in $V(G)$ under the isomorphism. In order to bound the enumeration for such ${T}$, we use the following two-step procedure to choose ${T}$:
\begin{itemize}
	\item  choose a subgraph ${T}_0\cong \mathbf{Span}(\mathbf R)$ of ${G}$ with leaves in $\mathbf R$ mapped to a fixed subset in $V(G)$ under the isomorphism;
	\item choose ${T}\cong \mathbf T$ of ${G}$ with $\mathbf{Span}(\mathbf R)$ mapped to $T_0$ under the isomorphism.
\end{itemize} 
Thus, by Definition~\ref{def:good-event} and by \eqref{eq-moderate-degree-condition}, the number of choices for such $T$ is $O\big(\mathbb D(\mathbf{Span}(\mathbf R))\times (np_\eta)^{|E(\mathbf T)\setminus E(\mathbf{Span}(\mathbf R))|}\big)$ on $\Gc$, and thus Item (ii) of $\Gc_s^2$ holds with probability $1-o(1)$. This completes the proof of Proposition~\ref{prop:good-events-are-typical}. 
\end{proof}

\end{document}